\documentclass
[reqno,fleqn,12pt]
{amsart}
\usepackage[utf8]{inputenc}

\allowdisplaybreaks

\usepackage{amsmath,amssymb,amsthm,mathrsfs,color,times,textcomp,verbatim,esint}
\usepackage{xcolor}
\usepackage[colorlinks=true]{hyperref}
\hypersetup{urlcolor=blue, citecolor=red, linkcolor=blue}
\usepackage[square,numbers]{natbib}

\numberwithin{equation}{section}
\theoremstyle{plain}

\newtheorem{theorem}{Theorem}[section]
\newtheorem{prop}[theorem]{Proposition}
\newtheorem{lem}[theorem]{Lemma}

\newtheorem{cor}[theorem]{Corollary}

\usepackage{graphicx}

\theoremstyle{definition}

\newtheorem{rem}[theorem]{Remark}

\newcommand{\R}{{\mathbb R}}
\newcommand{\C}{{\mathbb C}}
\newcommand{\Z}{{\mathbb Z}}

\newcommand{\e}{\epsilon}

\DeclareMathOperator{\re}{{\mathfrak{Re}}}
\DeclareMathOperator{\im}{{\mathfrak{Im}}}

\DeclareMathOperator{\dist}{dist}
\newcommand{\loc}{{\rm loc}}

\def\m{\boldsymbol{m}}

\def\bv{\boldsymbol{v}}

\def\Xint#1{\mathchoice
	{\XXint\displaystyle\textstyle{#1}}%
	{\XXint\textstyle\scriptstyle{#1}}%
	{\XXint\scriptstyle\scriptscriptstyle{#1}}%
	{\XXint\scriptscriptstyle\scriptscriptstyle{#1}}%
	\!\int}

\def\XXint#1#2#3{{\setbox0=\hbox{$#1{#2#3}{\int}$}
		\vcenter{\hbox{$#2#3$}}\kern-.5\wd0}}

\begin{document}

\title[The conformal limit for bimerons]
{The conformal limit for bimerons
\\ in easy-plane chiral magnets}
\author{Bin Deng}
\address{Institut de Math\'ematiques de Toulouse, UMR 5219, CNRS, UPS IMT, Universit\'e de Toulouse,
31062 Toulouse Cedex 9, France}
\email{bin.deng@math.univ-toulouse.fr}
\author{Radu Ignat}
\address{Institut de Math\'ematiques de Toulouse, UMR 5219, CNRS, UPS IMT, Universit\'e de Toulouse,
31062 Toulouse Cedex 9, France}
\email{radu.ignat@math.univ-toulouse.fr}
\author{Xavier Lamy}
\address{Institut de Math\'ematiques de Toulouse, UMR 5219, CNRS, UPS IMT, Universit\'e de Toulouse,
31062 Toulouse Cedex 9, France}
\email{xavier.lamy@math.univ-toulouse.fr}

\begin{abstract}
We study minimizers 
$\boldsymbol{m}\colon \mathbb R^2\to\mathbb S^2$ of the energy functional
\begin{align*}
E_\sigma(\boldsymbol{m})
=
\int_{\mathbb R^2}
\bigg(\frac 12 |\nabla\boldsymbol{m}|^2
+\sigma^2  \boldsymbol{ m} \cdot \nabla
\times\boldsymbol{m} 
+\sigma^2 m_3^2
 \bigg)\, dx\,,
\end{align*}
for $0<\sigma\ll 1$,
with prescribed topological degree
\begin{align*}
Q(\boldsymbol{m})=\frac{1}{4\pi}
\int_{\mathbb R^2}\boldsymbol{m} \cdot \partial_1 \boldsymbol{m}\times\partial_2\boldsymbol{m}\, dx =\pm 1\,.
\end{align*}
This model arises in
 thin ferromagnetic films with  Dzyaloshinskii-Moriya interaction and easy-plane anisotropy,
 where these minimizers represent \textit{bimeron} configurations.  
We prove their existence,
and describe them precisely as
 perturbations of specific Möbius maps:  
 we establish in particular that they are localized at scale of order $1/|\ln(\sigma^2)|$.
The proof follows a strategy introduced by Bernand-Mantel, Muratov and Simon (Arch. Ration.
Mech. Anal., 2021)
for
a similar model with easy-axis anisotropy,
but requires several adaptations to deal with the less coercive easy-plane anisotropy 
 and different symmetry properties.
\end{abstract}

\date{\today }
\subjclass[2010]{}
\keywords{}

\maketitle

\section{Introduction}

\subsection{Energy functional and topological degree}

For maps 
\begin{align*}
\m=(m_1,m_2,m_3)\colon\R^2\to\mathbb S^2\subset \R^3\,,
\end{align*}
 and
$\sigma>0$, we consider the energy
\begin{align}
\label{eq:Esigma}
E_{\sigma}(\m)
&
=D(\m) +\sigma^2 \big(  A(\m) + \widetilde H(\m) \big),
\\
    D(\m)
&
=\frac 12 \int_{\R^2}|\nabla \m|^2\, dx\,,
\qquad
A(\m)
=\int_{\R^2} m_3^2\, dx\,,
\nonumber
\\
\widetilde H(\m)
&
=2\int_{\R^2}m_3(\partial_1 m_2-\partial_2 m_1)\, dx\,.
\nonumber
\end{align}
It arises in the description of 
a thin ferromagnetic film 
with
 Dzyaloshinskii-Moriya interaction (DMI) and easy-plane anisotropy (see e.g. \cite{bachmann2023prb,gobel2019magnetic}).
 The map $\m$ represents the magnetization,
the Dirichlet term $D(\m)$ corresponds to the exchange energy,
and the term $A(\m)$ to easy-plane anisotropy favoring the horizontal plane $\{m_3=0\}$.
The DMI term $\widetilde H(\m)$
is well-defined as soon as the two other terms are finite.
Moreover, for $0<\sigma<1/2$, the energy density 
\footnote{In our results, $\sigma$ will be a small positive parameter $0<\sigma\ll1$.}
\begin{align*}
e_\sigma(\m)
&
=
\frac 12|\nabla \m|^2 + \sigma^2 m_3^2 +2\sigma^2 m_3(\partial_1 m_2-\partial_2 m_1)
\\
&
\geq 
\frac {1-2\sigma}{2}
|\nabla \m|^2 + \sigma^2(1-2\sigma) m_3^2 \ge 0\,,
\end{align*}
 is integrable if and only if $D(\m)+A(\m)<\infty$.

If a map $\m\colon\R^2\to\mathbb S^2$ is continuous and has a limit as $|x|\to +\infty$, then it can be identified 
with a continuous map $\tilde \m = \m\circ\Phi^{-1}\colon\mathbb S^2\to\mathbb S^2$, where $\Phi\colon\R^2\cup\lbrace\infty\rbrace
\approx\C\cup\lbrace\infty\rbrace
\to\mathbb S^2$ is the inverse stereographic projection
\begin{align}\label{eq:stereo}
    \Phi(z)=\left(\frac{2z}{1+|z|^2},\frac{|z|^2-1}{1+|z|^2}\right)
    \qquad\forall z\in \mathbb C\cup \lbrace\infty\rbrace\,.
\end{align}
The continuous map $\tilde\m\colon\mathbb S^2\to\mathbb S^2$ carries a topological degree,
which can also be defined if  $\tilde \m\in H^1(\mathbb S^2;\mathbb S^2)$ (see \cite{brezis1995degree}), 
and which characterizes the homotopy class of $\tilde\m$.
In terms of the original map $\m$, this corresponds to the topological degree
\begin{align}\label{eq:Q}
Q(\m)=\frac{1}{4\pi}\int_{\R^2}\m\cdot (\partial_1\m \times \partial_2 \m)\, dx
\in\mathbb Z\,.
\end{align}
see more details in Appendix~\ref{a:crit_sob_plane}, in particular   Corollary~\ref{c:Q}.
The purpose of this article is to describe,
for $0<\sigma\ll 1$,
minimizers  of $E_\sigma$ with unit degree $Q(\m)=\pm 1$, called \emph{bimerons} \cite{bachmann2023prb}.
The precise functional setting will be presented in \S~\ref{ss:statement}.

An analogous question is analysed in \cite{doring2017cvpde,bernand2021arma}
 for a model with easy-axis anisotropy, where the minimizes stand for \emph{skyrmions}.
The results and proofs share similarities, but also many differences and additional difficulties,
and we will carefully compare them after the statement of our main result.
For the easy-plane model considered here, 
results complementary to ours, dealing with bounded domains and a wider range of parameters, are obtained in \cite{bacho-melcher}.
 In the context of bounded domains, boundary magnetic vortices are also studied in the presence of the DMI term, see \cite{IgOf}.

\begin{rem}\label{r:DMI}
For maps $\m\colon\R^2\to\mathbb S^2$ such that 
$D(\m) +A(\m)<\infty$ and
$m_3$ decays   sufficiently fast at $\infty$,
the DMI term $\widetilde H(\m)$ in \eqref{eq:Esigma}  coincides with the classical expression
\begin{align*}
H(\m)
=\int_{\R^2} \m\cdot \nabla\times \m\, dx
=\lim_{R\to\infty}\int_{B_R} \m\cdot \nabla\times \m \, dx \,,
\end{align*}
where $\nabla$ is identified with $(\partial_1,\partial_2,0)$,
 thanks to the identity
\begin{align*}
&
\m\cdot\nabla\times \m - 2m_3(\partial_1 m_2-\partial_2 m_1)
\nonumber 
= 
-\partial_1(m_3m_2) +\partial_2(m_3m_1).
\end{align*}
Under the mere condition that $D(\m)+ A(\m)<\infty$,
we do have
\begin{align*}
\lim_{R_k\to\infty}\int_{B_{R_k}} \m\cdot\nabla\times \m \, dx =\widetilde H(\m)\,,
\end{align*} 
along a sequence $R_k\to\infty$, 
but $H(\m)$ might not be well-defined and the full limit as $R\to +\infty$ might fail to exist. 
\end{rem}

\subsection{Symmetries}

Several groups of geometric transformations play an important role in our analysis: dilations
\begin{align}\label{eq:Drho}
\mathfrak D_\rho \m(x)=\m\Big(\frac x\rho\Big)\,,\quad \rho>0\,,
\end{align}
translations
\begin{align}\label{eq:Tx0}
\mathfrak T_{x_0}\m (x)=\m(x-x_0)\,,\quad  x_0\in\R^2\,,
\end{align}
and corotations
\begin{align}
\mathfrak R_\phi \m(x) 
=R_{\mathbf e_3,\phi}\m(e^{-i\phi}x)\,,
\quad \phi\in\R\,,
\label{eq:Rphi}
\end{align}
where $R_{\mathbf e_3,\phi}\in SO(3)$ is the rotation of axis $\mathbf e_3$ and angle $\phi$ in $\R^3$, and $e^{-i\phi}$ is the rotation 
of angle $-\phi$ in $\R^2\approx\C$.

Dilations have different effects on each energy term in \eqref{eq:Esigma} namely,
for any $\rho>0$,
\begin{align}
D(\mathfrak D_\rho\m)
&
=D(\m)\,,
\nonumber
\\
\widetilde H(\mathfrak D_\rho \m)
&
=\rho \widetilde H(\m)\,,
\nonumber
\\
A(\mathfrak D_\rho \m)
&
=\rho^2 A(\m)\,,
\label{eq:rescaling}
\end{align}
and the topological degree \eqref{eq:Q} is invariant under dilations:
\begin{align*}
 Q(\mathfrak D_\rho \m)=Q(\m)\,.
\end{align*}
As for translations and corotations, they keep both energy and topological degree invariant:
\begin{align}
&
E_\sigma(\mathfrak T_{x_0}\m)=E_\sigma(\mathfrak R_{\phi}\m)=E_\sigma(\m)\,,
\label{eq:Tx0Rphi}
\\
&
Q(\mathfrak T_{x_0}\m)=Q(\mathfrak R_{\phi}\m)=Q(\m)\,,
\quad\forall x_0\in\R^2\,,\forall \phi\in\R\,.
\nonumber
\end{align}
Finally we also note the effect of the
reflection $\m\leadsto -\m$, which keeps the energy invariant but reverses the degree:
\begin{align}\label{eq:reflection}
E_\sigma(-\m)=E_\sigma(\m),\quad Q(-\m)=-Q(\m)\,.
\end{align}
Thanks to this property, we reduce the study of minimizers under the topological constraint $Q(\m)=\pm 1$, to the case $Q(\m)=-1$.

\begin{rem}\label{r:choice_coeff}
One consequence of the scaling properties \eqref{eq:rescaling}  is that 
the specific choice of coefficients in front of each term in \eqref{eq:Esigma} is not restrictive: for any $\lambda,\rho >0$ we have 
\begin{align*}
\lambda E_\sigma(\mathfrak D_\rho \m)
&=
\lambda D(\m) +\lambda\rho\sigma^2\widetilde H(\m) +\lambda\rho^2\sigma^2 A(\m)\,,
\end{align*}
so any result about $E_\sigma$ 
can be translated into a result about an energy with arbitrary coefficients $a=\lambda$, $b=\lambda\rho\sigma^2$, $c=\lambda\rho^2\sigma^2>0$ in front of the three energy terms.
With these general coefficients, the regime $\sigma^2\ll 1$ considered in this article corresponds to $b^2\ll ac$. 
\end{rem}

\subsection{The conformal limit: heuristic description}\label{ss:conf_lim}

In the conformal limit $\sigma\to 0$, the energy $E_\sigma$ in \eqref{eq:Esigma} formally reduces to the Dirichlet energy $D(\m)$,
whose minimizers with prescribed degree $Q(\m)$ are well-known and satisfy $D(\m)=4\pi|Q(\m)|$ (see
 e.g. \cite{lemaire1978jdg}).
More precisely, the pointwise inequality 
\begin{align*}
\big| \m\cdot (\partial_1\m\times\partial_2 \m) \big| \leq \frac{1}{2}|\nabla \m|^2\,,
\end{align*}
implies indeed that
\begin{align}\label{eq:DQ}
D(\m)\geq 4\pi |Q(\m)|,
\end{align}
with equality if and only if $\partial_1 \m\cdot\partial_2 \m = |\partial_1\m|^2-|\partial_2 \m|^2=0$, that is, $\m$ is conformal.
Conformal maps of degree $Q(\m)=-1$ are given by the Möbius group
\begin{align}\label{eq:mobmaps}
\mathcal M
&
=\Big\lbrace
\Phi\Big(\frac{az +b}{cz +d}\Big) \colon a,b,c,d\in\C,\, ad-bc\neq 0
\Big\rbrace\,,
\end{align}
where $\Phi\colon \C\cup\lbrace\infty\rbrace \to\mathbb S^2$ is the inverse stereographic projection defined in \eqref{eq:stereo}.
In the parametrization \eqref{eq:mobmaps}, the determinant $ad-bc\in\C \setminus \{0\}$ can always be fixed, 
and the Möbius group is
 six-dimensional (as a real manifold).

Formally, the constraint $A(\m)<\infty$ forces $m_3\to 0$ at $\infty$.
For small positive $\sigma$,
 one therefore expects minimizers of $E_\sigma$ of degree $-1$
  to be close to conformal maps with $|a|=|c|$ in the parametrization \eqref{eq:mobmaps}.
Another real parameter can be fixed by minimizing the DMI term,
and this leaves a four-real-parameter family of conformal maps.
Three of these parameters come from the translational and corotational invariance of the energy \eqref{eq:Tx0Rphi}.
The last parameter can be interpreted as a scaling parameter $\rho>0$, corresponding to a dilation \eqref{eq:Drho}.
Performing these reductions explicitly, one 
 expects
 that a minimizer $\m_\sigma$ of $E_\sigma$ with $Q(\m_\sigma)=-1$
should
 be close to a Möbius map 
\begin{align*}
\Psi (z)=\Phi\Big( w_*\Big(\frac{z}{\rho}\Big)\Big),\qquad
w_*(z)=i\frac{z-1}{z+1}\,,
\end{align*}
modulo translation \eqref{eq:Tx0} and corotation \eqref{eq:Rphi}.
The expression of $w_*(z)$ features a vortex at $z=1$ and an antivortex at $z=-1$, characteristic of bimeron structures, as described in \cite{bachmann2023prb} 
(see  \cite{Hang-Lin} or \cite[\S~7.2]{Ignat} for a survey about the mathematical analysis of magnetic vortices).

The three energy terms 
have different scaling behaviors \eqref{eq:rescaling},
so
one further expects the scale of concentration $\rho>0$ to be fixed by the competition between these three terms.
However, 
since Möbius maps $\Psi\in\mathcal M$ have infinite anisotropy $A(\Psi)=\infty$,
one cannot simply plug this ansatz into the energy:
identifying that concentration scale $\rho$ is a more subtle task.
For skyrmions, this task was carried out in \cite{bernand2021arma}
by introducing several powerful new tools and ideas.
Here we adapt that strategy to bimerons, 
and obtain that minimizers $\m_\sigma$ of $E_\sigma$ with topological degree
 $Q(\m_\sigma)=-1$ concentrate at scale
 \begin{align*}
 \rho_\sigma =\frac{1+o(1)}{\ln(1/\sigma^2)}
 \quad\text{as }\sigma\to 0\,,
 \end{align*}
 that is, $\m_\sigma$ is close, in a sense to be made precise below, to the orbit of the Möbius map
 \begin{align*}
\Psi_\sigma(z)=\Phi\Big(w_*\Big(\frac z{\rho_\sigma}\Big)\Big),\quad w_*(z)=i\frac{z-1}{z+1}\,,
 \end{align*}
 under the action of translations \eqref{eq:Tx0} and corotations \eqref{eq:Rphi}.

\subsection{Functional framework and precise statement}\label{ss:statement}

For any measurable map $\m\colon\R^2\to\mathbb S^2$ belonging to the homogeneous Sobolev space
\begin{align}\label{eq:homHS2}
\mathcal H(\R^2;\mathbb S^2)
&=\bigg\lbrace
\m\in H^1_{\loc}(\R^2;\mathbb S^2)\colon \int_{\R^2}|\nabla \m|^2\, dx <\infty
\bigg\rbrace\,,
\end{align}
composing  with the  stereographic projection \eqref{eq:stereo} 
gives a map
\begin{align*}
\tilde \m = \m\circ\Phi^{-1} \in H^1(\mathbb S^2;\mathbb S^2)\,,
\end{align*}
whose topological degree is a well-defined integer \cite{brezis1995degree} and 
equal to the topological degree $Q(\m)$ defined in \eqref{eq:Q}.
See Appendix~\ref{a:crit_sob_plane} for more details about these claims.

For all $\sigma>0$ we
 consider 
 the energy $E_\sigma$ defined by  
 \eqref{eq:Esigma} on the set
\begin{align}
\label{eq:W}
\mathcal W
&
=\Big\lbrace
\m\in   \mathcal H(\R^2;\mathbb S^2) \colon  A(\m)<\infty
\Big\rbrace
\\
&
=\Big\lbrace
\m\in   \mathcal H(\R^2;\mathbb S^2) \colon  \int_{\R^2} m_3^2\, dx <\infty
\Big\rbrace\,,
\nonumber
\end{align} 
which the value of the topological degree partitions into the subsets
\begin{align}\label{eq:Wq}
\mathcal W_q =\Big\lbrace \m \in \mathcal W \colon Q(\m)=q\Big\rbrace\,,\quad\text{for }
q\in\Z\,.
\end{align} 
Recall that the reflection
$\m\leadsto -\m$ provides a bijection between $\mathcal W_{q}$ and $\mathcal W_{-q}$, while  preserving the energy \eqref{eq:reflection}.
With these notations, our main result provides precise asymptotics 
for minimizers of $E_\sigma$ on $\mathcal W_{-1}$
in the limit $\sigma\to 0$.

\begin{theorem}\label{t:main}
There exist absolute constants $\sigma_0, C>0$ such that,
for any $\sigma\in (0,\sigma_0]$,
the infimum of $E_\sigma$ over $\mathcal W_{-1}$ is attained, 
 satisfies
\begin{align*}
\min_{\mathcal W_{-1}}E_\sigma 
=4\pi 
-
\frac{2\pi\sigma^2}{\ln\big((1/\sigma^2)\ln^2(1/\sigma^2)\big)}
+\mathcal O\Big(\frac{\sigma^2}{\ln^2(1/\sigma^2)}\Big)
\,,
\end{align*}
and, for any minimizing map $\m_\sigma\in \mathcal W_{-1}$,
there exist
  $\rho_\sigma>0$ and $\alpha_\sigma\in\R$ estimated by
\begin{align*}
&
\big| \ln(1/\sigma^2)\,\rho_\sigma-
1\big| +
|\alpha_\sigma|\leq \frac{C}{\sqrt{\ln(1/\sigma^2)}}
\,,
\end{align*}
and a Möbius map $\Psi_\sigma\in\mathcal M$ characterized by
 \begin{align}
 \label{eq:Mob_orbit}
\mathfrak T_{z_\sigma}\mathfrak R_{\phi_\sigma}\Psi_\sigma(z)
&
=\Phi\Big( w_*\Big(e^{-i\alpha_\sigma}\frac z{\rho_\sigma}\Big)\Big)
\,,
\quad
w_*(z)
=
i\frac{z-1}{z+1}\,,
\end{align}
for some translation 
 $\mathfrak T_{z_\sigma}$ and corotation $\mathfrak R_{\phi_\sigma}$ as in \eqref{eq:Tx0}-\eqref{eq:Rphi}, 
 such that
\begin{align*}
 \int_{\R^2}
\big|\nabla
\big(
\m_\sigma -\Psi_\sigma\big)\big|^2\, dx 
 \leq C\frac{\sigma^2}{\ln  (1/\sigma^2)}\,.
 \end{align*}
\end{theorem}

\begin{rem}
A slightly more precise description of the orbit of Möbius maps closest to minimizers of $\mathcal E_\sigma$ with degree $Q=-1$ is given in Proposition~\ref{p:lowE}.
Its proof also makes the contribution of each energy term apparent:
\begin{align*}
D(\m_\sigma)
&
=4\pi  +\mathcal O\Big(\frac{\sigma^2}{\ln^2(1/\sigma^2)}\Big)\,,
\\
H(\m_\sigma)
&
=-2A(\m_\sigma)
= - \frac{4\pi}{\ln\big((1/\sigma^2)\ln^2(1/\sigma^2)\big)}
+\mathcal O\Big(\frac{1}{\ln^2(1/\sigma^2)}\Big)\,.
\end{align*}
Here, the Pohozaev identity $H(\m_\sigma)=-2A(\m_\sigma)$ comes from 
criticality of the minimizer $\m_\sigma$ with respect to scaling.
\end{rem}

\subsection{Comparison with skyrmions}\label{ss:skyrmions}

The main difference between bimerons and skyrmions is that the easy-plane anisotropy $A(\m)$ in the energy \eqref{eq:Esigma} is replaced by an easy-axis anisotropy
\begin{align}
A^{\text{easy-axis}}_1(\m)
&=\int_{\R^2}(1-m_3^2)\, dx\,,
\label{eq:easyaxis}
\\
\text{or}
\quad
A_2^{\text{easy-axis}}(\m)
&=\int_{\R^2}(1-m_3)\, dx\,.
\nonumber
\end{align}
The first version is used in \cite{bernand2021arma}, and the second  in \cite{melcher2014proc,doring2017cvpde}.
Formally, the first enforces only $\m(x)\to\pm\mathbf e_3$ as  $|x|\to\infty$, 
and the second selects the orientation $+\mathbf e_3$.
In practice, minimizers for the first version are obtained in a space which already selects an orientation, see e.g. \cite[\S~2.1]{bernand2021arma}, and the analysis of these two different easy-axis models is extremely similar.
We will not comment on 
additional nonlocal energy terms which are also considered in \cite{bernand2021arma}.

\begin{rem}\label{r:DMI_easyaxis}
Changing the anisotropy  forces to change the DMI term.
Let $m' =(m_1,m_2)$ and $\nabla^\perp =(-\partial_2,\partial_1)$. 
If $D(\m)+A^{\text{easy-axis}}_j(\m)<\infty$ for $j=1$ or $2$,
 the DMI term
\begin{align*}
\widetilde H^{\text{easy-axis}}(\m)=-2\int_{\R^2}m' \cdot\nabla^\perp m_3\, dx\,,
\end{align*} 
is well-defined,
since $|m' |^2 =1-m_3^2$ is controlled by the easy-axis anisotropy.
Similarly to Remark~\ref{r:DMI}, it satisfies
\begin{align*}
\widetilde H^{\text{easy-axis}}(\m)
=
\lim_{k\to\infty}\int_{B_{R_k}} \m\cdot\nabla\times\m\, dx\,,
\end{align*}
for some sequence $R_k\to \infty$.
Under the condition
$D(\m)+A_2^{\text{easy-axis}}(\m)<\infty$
which selects the orientation $+\mathbf e_3$ at $\infty$,
this DMI term can also be rewritten as
 \begin{align}
\widetilde H^{\text{easy-axis}}(\m)
&
=
\int_{\R^2}(\m-\mathbf e_3)\cdot \nabla\times \m\, dx
\nonumber
\\
&
=2\int_{\R^2}(m_3-1)(\partial_1 m_2-\partial_2 m_1)\, dx\,,
\label{eq:DMI_easyaxis}
\end{align}
as used repeatedly in \cite{melcher2014proc,doring2017cvpde,bernand2021arma}.
\end{rem}

The main, obvious difference between 
our easy-plane anisotropy $A(\m)$ and the easy-axis anisotropy $A_1^{\text{easy-axis}}(\m)$ 
is the structure of admissible constant states, 
or, equivalently, 
 admissible far-field limits:
when it exists, the limit
 $\m_\infty =\lim_{|x|\to\infty} \m(x) \in\mathbb S^2$
 must satisfy
\begin{align*}
&
\m_\infty \in \mathbb S^1\times\lbrace 0\rbrace
&&
\text{if }A(\m)<\infty,\\
\text{versus}
\quad
&
\m_\infty \in \lbrace \pm \mathbf e_3\rbrace
&&
\text{if }A_1^{\text{easy-axis}}(\m)<\infty\,.
\end{align*}
We list next several consequences of this difference.

 \medskip
 
\noindent\textbf{Sign of the topological degree.} 
The set $\mathbb S^1\times\lbrace 0\rbrace$ is connected,
while $\lbrace\pm \mathbf e_3\rbrace$ has two  components.
They divide each homotopy class of $\mathbb S^2$-valued maps into two distinct classes of admissible easy-axis maps,
causing a distinction between positive and negative topological degrees in the skyrmion model \cite{melcher2014proc,doring2017cvpde,bernand2021arma}.
This distinction is not present here.

\medskip
\noindent\textbf{Symmetric solutions.}
The corotational symmetry of the energy \eqref{eq:Tx0Rphi},
also valid in the easy-axis model,
 makes it natural
 to look for critical points which are axisymmetric: invariant under the action of corotations \eqref{eq:Rphi},
 that is, $\m(x) =R_{\mathbf e_3,\phi} \m(e^{-i\phi} x)$ for all $\phi\in\R$.
In polar coordinates $z=re^{i\theta}$,
axisymmetric maps are of the form
$\m_{\mathrm{sym}}(re^{i\theta})=R_{\mathbf e_3,\theta}\m(r)$.
This ansatz is compatible with $\m_\infty\in \lbrace\pm \mathbf e_3\rbrace$ in the easy-axis case, and 
axisymmetric skyrmions of degree $Q=-1$ are analysed in \cite{li2018jfa}.
But 
there are no axisymmetric bimerons,
because this ansatz is not compatible 
with $\m_\infty\in\mathbb S^1\times\lbrace 0\rbrace$, 
and more generally with $D(\m)+A(\m)<\infty$.
Indeed, 
 for such a symmetric ansatz,
 we have 
\begin{align*}
&
\int_1^\infty (1-m_3^2)\frac{dr}{r} =\frac{1}{2\pi}\int_{|x|\geq 1}\frac{|\partial_\theta\m_{\mathrm{sym}}|^2}{r^2}\, dx  \leq
 D(\m_{\mathrm{sym}})\,,
\\
\text{and }
&
\int_{1}^\infty m_3^2\frac{dr}{r} \leq \int_1^\infty m_3^2 \, r dr  + \int_{1}^{\infty} m_3^2\frac{dr}{r^3} \leq A(\m_{\mathrm{sym}}) +1\,,
\end{align*}
hence $D(\m_{\mathrm{sym}})+A(\m_{\mathrm{sym}})<\infty$ would imply $\int_1^\infty dr/r <\infty$, a contradiction.

\medskip
\noindent
\textbf{Dimension of the selected Möbius maps.}
In the easy-plane case described by Theorem~\ref{t:main}, 
the selected orbit of Möbius maps (under the action of translations and corotations) is three-dimensional,
while for skyrmions it is two-dimensional.
This is related 
with the previous observation: the skyrmions' orbit is smaller because it contains an axisymmetric map, which stays fixed under the action of corotations \eqref{eq:Rphi}.

\medskip
\noindent
\textbf{Far-field behavior.}
The set $\mathbb S^1\times\lbrace 0\rbrace$ is one-dimensional, while
$\lbrace\pm \mathbf e_3\rbrace$ is discrete.
In that sense, easy-axis anisotropy is much more constraining,
and this is reflected in the far-field behavior of finite-energy configurations, that is, their behavior as $|x|\to +\infty$.
The assumption of finite easy-axis anisotropy $A_j^{\text{easy-axis}}(\m)<\infty$
(for $j=1$ or 2)
implies that $\infty$ is a Lebesgue point of $\m$, 
in the sense that
\begin{align*}
\m_R:=\Xint{-}_{|x|\geq R}\m \, \frac{dx}{(1+|x|^2)^2} \to \pm\mathbf e_3\quad\text{as }R\to\infty\,,
\end{align*} 
see  Corollary~\ref{c:lebesgue_infty}.
Here instead, under the finite easy-plane assumption $A(\m)<\infty$,
we only know that
$\dist(\m_R,\mathbb S^1\times\lbrace 0\rbrace)\to 0$,
and $\m_R$ might fail to have a limit as $R\to\infty$.
Consider for instance $\m(x)=\Phi(e^{i\varphi(x)}w_*(x))$ with $w_*$ as in Theorem~\ref{t:main} and $\varphi(x)=\ln(1+\ln(1+|x|^2))$, 
then $\m\in \mathcal W_{-1}$ 
 but $\infty$ is not a Lebesgue point of $\m$.

\medskip
\noindent
\textbf{Control of the DMI term.}
 As another consequence of the previous point, 
the easy-axis anisotropy provides a more efficient control on the DMI term: indeed, using the expression \eqref{eq:DMI_easyaxis} we see that
\begin{align*}
\big|
\widetilde H^{\text{easy-axis}}(\m)
\big|^2
&
\leq 16 D(\m)  \int_{\R^2}(1-m_3)^2\, dx \,,
\end{align*}
and  the last integral is controlled by 
$A_2^{\text{easy-axis}}(\m)$, 
but, near $\infty$, its integrand is actually much smaller than the integrand $(1-m_3)$ of $A_2^{\text{easy-axis}}(\m)$.
This improved control is used crucially in \cite{melcher2014proc} and \cite{bernand2021arma}, but is 
 absent in our case.

\subsection{Proof ideas}\label{ss:ideas}

As explained above, the proof of Theorem~\ref{t:main} 
relies primarily on adapting  the strategy introduced in \cite{bernand2021arma}
for a similar model with easy-axis anisotropy \eqref{eq:easyaxis}.
The main tool is 
a stability estimate \cite[Theorem~2.4]{bernand2021arma}
(see  \cite{hirsch2022,topping2023} for alternative proofs and
 \cite{rupflin23} for a generalization to higher degrees)
 which implies that minimizers of $E_\sigma$ with degree $Q=-1$ must be close to the Möbius group \eqref{eq:mobmaps}.
This information can then be used to obtain lower bounds on each energy term,
which depend on the closest Möbius map.
Comparing these lower bounds with
energy competitors 
 which are close to the optimal orbit of Möbius maps \eqref{eq:Mob_orbit} then implies
that the closest Möbius map must belong to a neighborhood of that
optimal orbit.

Next we underline the main new elements in our proof, compared to
the analysis performed in \cite{bernand2021arma} (and also \cite{melcher2014proc}).

\begin{itemize}
\item The classical parametrization \eqref{eq:param_classic} of the Möbius group \eqref{eq:mobmaps}, 
was convenient in \cite{bernand2021arma}
to describe the Möbius maps close to skyrmions,
but is not well adapted to the orbit of Möbius maps \eqref{eq:Mob_orbit} close to bimerons.
We provide a new and more adapted parametrization in Lemma~\ref{l:param_mobius}.

\item 
As explained in the last point of \S~\ref{ss:skyrmions}, 
the control on the DMI term $H(\m)$ provided by the easy-plane 
anisotropy $A(\m)$ is less coercive than that
 in the easy-axis case.
This plays a role in two places, 
where we cannot use the arguments in \cite{melcher2014proc} and \cite{bernand2021arma}:
 to obtain 
a sharp lower bound on the DMI term, and to rule out `vanishing' in the proof of existence.
We introduce new arguments to circumvent that lack of efficient control: see \S~\ref{ss:lowH}, and Step~3 in the
proof of Proposition~\ref{p:existence}.
\item
 The upper bound is obtained by a construction which relies on modifying Möbius maps to make their anisotropy $A(\m)$ finite.
Here we use, as in  \cite{doring2017cvpde}, a basic cut-off construction at one scale.
The construction in \cite{bernand2021arma} is more elaborate and cuts tails off in an optimal way with a modified Bessel function.
The two constructions turn out to provide the same accuracy, see Remark~\ref{r:truncation}.
It only affects the explicit bound on the remainder term of order $\sigma^2/\ln^2\sigma$ in the energy expansion.
 \end{itemize}

\subsection{Plan of the article}

The article is organized as follows. In \S~\ref{s:param_mob} we describe our tailored parametrization of the Möbius group.
In \S~\ref{s:DMIconf} we calculate the DMI energy of conformal maps.
In \S~\ref{s:up} we 
describe the construction which provides the energy upper bound.
In \S~\ref{s:low} we prove the lower bound and characterize maps which almost saturate it.
In \S~\ref{s:exist} we use that characterization to prove existence of minimizers for $0<\sigma\ll 1$ and conclude the proof of Theorem~\ref{t:main}.

\subsection*{Acknowledgments}
The authors are supported in part by the ANR projects ANR-22-CE40-0006 and  ANR-21-CE40-
0004, and by LabEx CIMI.

\section{A parametrization of the Möbius group} \label{s:param_mob}

The Möbius group \eqref{eq:mobmaps} can be parametrized by
\begin{align}\label{eq:param_classic}
\m=S\Phi\left( \frac{z-z_0}{\rho} \right),\qquad z_0\in\C, \, \rho>0,\, S\in SO(3).
\end{align}
This parametrization
 was convenient 
for the study of skyrmions in \cite{bernand2021arma}.
It is naturally  expressed in terms of the translation operators $\mathfrak T_{z_0}$ defined in \eqref{eq:Tx0}, 
and of the dilation operators $\mathfrak D_\rho$ defined in \eqref{eq:Drho}.
The corotation operators $\mathfrak R_\phi$ defined in \eqref{eq:Rphi} do not appear, 
which is consistent with the fact that the optimal orbit of Möbius maps closest to skyrmions contains an axisymmetric map, as explained in \S~\ref{ss:skyrmions}.
But in our case we need to keep better track of corotations,
and we choose therefore a different parametrization.

\begin{lem}
\label{l:param_mobius}
The Möbius group can be parametrized by
\begin{align}\label{eq:param_mobius}
\m^{\lbrace z_0,\rho,\phi,\alpha,\beta\rbrace} 
&
=
\mathfrak T_{z_0}
\mathfrak D_\rho
\mathfrak R_\phi 
\m^{[\alpha,\beta]},
\\
& z_0\in\C,\, \rho>0,\,\phi,\alpha,\beta\in\R\,,
\nonumber
\end{align}
where 
\begin{align*}
\m^{[\alpha,\beta]}
&
=R_{\mathbf e_1,2\beta}\Phi(w_*(e^{-i\alpha}z)) =\Phi(w^{[\alpha,\beta]}(z)),
\\
w^{[\alpha,\beta]}(z)
&
=
\frac{\cos( \beta  ) w_* (e^{-i\alpha}z) +i\sin( \beta  )}
{i\sin( \beta  )w_*(e^{-i\alpha}z)+\cos( \beta )},
\quad
w_*(z)=i\frac{z-1}{z+1}\,.
\end{align*}
\end{lem}

\begin{rem}\label{r:walphabetaId}
For $\beta=\pi/4$, we have $w^{[\alpha,\pi/4]}(z)=e^{i(\frac\pi 2 -\alpha)}z$.
\end{rem}

\begin{rem}\label{r:betaw*}
For $\beta\in\frac{\pi}{2}\Z$ we have $w^{[0,\beta]}=w_*$ or $1/w_*$ depending on the value of $\beta$ modulo $\pi/2$, 
but these two maps are on the same orbit generated by corotations: for $\phi=\pi$ we have $\mathfrak R_\pi[ \Phi(w_*)] = \Phi(1/w_*)$.
\end{rem}

\begin{proof}[Proof of Lemma~\ref{l:param_mobius}]
The maps \eqref{eq:param_mobius} are clearly Möbius maps, 
we need to check that all Möbius maps can be obtained this way.

Consider an arbitrary Möbius map $\m$. There exists $\phi\in\R$ such that 
$\mathbf v = R_{\mathbf e_3,-\phi}\m(\infty)$ belongs to $\mathbf e_1 ^\perp\cap\mathbb S^2$, the large circle which contains $\mathbf e_2$ and $\mathbf e_3$.
Then there exists $\beta \in \R$ such that $R_{\mathbf e_1,-2\beta}\mathbf v=\mathbf e_2$, so that 
\begin{align*}
{\tilde \m}=R_{\mathbf e_1,-2\beta}\mathfrak R_{-\phi} \m
\end{align*}
 satisfies
${\tilde \m}(\infty)=\mathbf e_2$.
As a consequence, ${\tilde \m}$ can be written as
\begin{align*}
{\tilde \m}(z)=\Phi\left(i \frac{z-a}{z+b} \right),
\end{align*}
for some $a,b\in \mathbb C$ such that $a+b\neq 0$.
We infer that
\begin{align*}
\m  = \mathfrak R_{\phi} R_{\mathbf e_1,2\sigma}\Phi(\tilde w),\quad \tilde w(z)=i\frac{z-a}{z+b}.
\end{align*}
Using the identities
\begin{align*}
    R_{\mathbf e_1,2\theta}\Phi(w)
    &
    =\Phi\left(\frac{w\cos(\theta) +i\sin(\theta) }{i\, w\sin(\theta) +\cos(\theta) }\right),
    \\
     R_{\mathbf e_3,\theta}\Phi(w)
     &
     =\Phi(e^{i\theta}w)\,,
\end{align*}
this becomes $\m=\Phi(w)$ with
\begin{align*}
w(z)
&
=
e^{i\phi}
\frac{\cos( \beta  ) \tilde w (e^{-i\phi}z) +i\sin( \beta  )}
{i\sin( \beta  )\tilde w(e^{-i\phi}z)+\cos( \beta )}
\,.
\end{align*}
Thus, for any $z_0\in\C$ and $\rho>0$ we have
\begin{align*}
\mathfrak D_{1/\rho}\mathfrak T_{-z_0} \m =\mathfrak R_\phi R_{\mathbf e_1,2\beta}\Phi(\bar w),
\end{align*}
with
\begin{align*}
\bar w(z)
&
=\tilde w (\rho z+e^{-i\phi}z_0)
=i\frac{z-(a-e^{-i\phi}z_0)/\rho}{z+(b+ e^{-i\phi}z_0)/\rho }.
\end{align*}
Choosing
\begin{align*}
\rho =\frac{|a+b|}{2},
\quad
z_0 =e^{i\phi}\frac{a-b}{2},
\end{align*}
we obtain
\begin{align*}
\bar w(z)=i\frac{z-e^{i\alpha}}{z+e^{i\alpha}}
=w_*(e^{-i\alpha}z),\qquad \alpha=\arg(a+b).
\end{align*}
We conclude that
\begin{align*}
\m=\mathfrak T_{z_0}\mathfrak D_\rho \mathfrak R_\phi \m^{[\alpha,\beta]},
\end{align*}
with $\m^{[\alpha,\beta]}(z)=R_{\mathbf e_1,2\beta}\Phi(w_*(e^{-i\alpha}z))$ 
and the parametrization \eqref{eq:param_mobius} is indeed surjective.
\end{proof}

\section{DMI energy of Möbius maps}\label{s:DMIconf}

In
this section we compute
 the DMI energy of 
 Möbius maps
 \begin{align*}
\m(z) &
=\Phi(w^{[\alpha,\beta]}(z))
=R_{\mathbf e_1,2\beta -\pi/2}\Phi(w^{[\alpha,\pi/4]}(z))
\\
&
=R_{\mathbf e_1,2\beta -\pi/2}\Phi(ie^{-i\alpha}z)\,,
\end{align*}
 where $\Phi$ 
 is the stereographic map defined in \eqref{eq:stereo}.
The equalities follow from the definition of 
 $w^{[\alpha,\beta]}$ 
 in Lemma~\ref{l:param_mobius}
 and from Remark~\ref{r:walphabetaId}.
For such map $\m$, 
the integrand of $\tilde H(\m)$ 
is not absolutely
 integrable,
but we will see that the integral
\begin{align}
    \widetilde H(\m;B_R)=2\int_{B_R}m_3
     (\partial_1 m_2-\partial_2 m_1)\, dx\,,
\end{align}
admits a limit as $R\to\infty$,
thus providing a meaningful definition of $\widetilde H(\m)$.
Using 
 the notation $z=x+i y\in \mathbb{C}\approx \mathbb{R}^2$,
we find
\begin{align*}
    m_1&=\frac{2 (x \sin (\alpha)-y \cos (\alpha))}{x^2+y^2+1},
    \\
    m_2&=\frac{\cos (2 \beta)
   \left(x^2+y^2-1\right)}{x^2+y^2+1}+\frac{2 \left(x \cos (\alpha) +y \sin (\alpha) \right) \sin (2 \beta)}{x^2+y^2+1},
   \\
   m_3&=\frac{ \sin (2 \beta)
   \left(x^2+y^2-1\right)}{x^2+y^2+1}-\frac{2 \left(x \cos (\alpha)+y \sin (\alpha)\right) \cos (2 \beta)}{x^2+y^2+1}
   \,,
\end{align*}
and
\begin{align*}
2m_3(\partial_1 m_2-\partial_2 m_1)
&
=
\frac{4 g(x,y)\, h(x,y)}{(1+x^2+y^2)^3}\,,
\end{align*}
where
\begin{align*}
    g(x,y)&=-\sin (2 \beta) +\left(x^2+y^2\right)\sin(2\beta) 
    \\
    &\quad -2 x \cos (\alpha)  \cos (2 \beta) -2 y \sin (\alpha)\cos (2 \beta),
    \\
    h(x,y)&=  \cos (\alpha) 
    \big(1+\sin (2 \beta)
    +(1-\sin(2\beta))( x^2-y^2)
\big)
   \\
   &\quad +2 x \cos (2 \beta)+2xy  \sin (\alpha) \left(1- \sin (2 \beta)\right)\,.
\end{align*}
Integrating on $B_R$,
and using that a function $f(x,y)$ 
has zero integral if it satisfies one of the antisymmetry
properties
 $f(x,-y)=-f(x,y)$ or $f(-x,y)=-f(x,y)$ or $f(y,x)=-f(x,y)$, we are left with
 \begin{align*}
 \frac 14
 \widetilde H(\m;B_R)
 &
 =\cos(\alpha)\sin(2\beta)(1+\sin(2\beta))
I_1 
 \\
 &
\quad
 -4\cos\alpha\cos^2(2\beta)
I_2
,
\end{align*}
where
\begin{align*}
I_1 &=
 \int_{B_R}\frac{x^2+y^2 -1}{(x^2+y^2+1)^3}\,dxdy
 =-\pi\frac{R^2}{(1+R^2)^2}\,,
\\
I_2
&
=\int_{B_R}\frac{x^2}{(x^2+y^2+1)^3}dxdy
=\frac\pi 4 \frac{R^4}{(1+R^2)^2}\,,
 \end{align*}
and therefore
\begin{align*}
\frac{(1+R^2)^2}{4\pi}\widetilde H(\m;B_R)
&
= -\cos\alpha\cos^2(2\beta) R^4
\\
&
\quad
-\cos(\alpha)\sin(2\beta)(1+\sin(2\beta))R^2\,.
\end{align*}
Taking the limit as $R\to\infty$, we deduce
\begin{align}\label{eq:DMImalphabeta}
\widetilde H(\m^{[\alpha,\beta]})
= \lim_{R\to\infty} \widetilde H(&\m^{[\alpha,\beta]};B_R)
=
-4\pi\cos(\alpha) \cos^2(2\beta)\,,
\end{align}
and we also obtain the estimate
\begin{align}\label{eq:DMImalphabetaR}
\big| 
\widetilde H(\m^{[\alpha,\beta]}) -\widetilde H(\m^{[\alpha,\beta]};B_R)
\big|
 \leq \frac{C}{R^2}\,,
\end{align}
for all $R\geq 2$, where $C>0$ is an absolute constant.

\section{Energy upper bound}\label{s:up}

In this section we prove the upper bound on the minimal energy, 
by estimating the energy of explicit
 competitors.

\begin{prop}\label{p:up}
The infimum of the energy $E_\sigma$ defined in \eqref{eq:Esigma} over the space $\mathcal W_{-1}$ defined in \eqref{eq:Wq} is bounded by
\begin{align}
\inf_{\mathcal W_{-1}}E_\sigma
&
\leq
4\pi 
-\frac{\pi\sigma^2}{\ln\big(\sigma^{-1}\ln(1/\sigma)\big)}
+C\frac{\sigma^2}{\ln^2\sigma}
\,,
\label{eq:up}
\end{align}
for some absolute constant $C>0$.
\end{prop} 
 
\begin{proof}
[Proof of Proposition~\ref{p:up}]
As explained in the introduction,
 one would like to use Möbius maps as competitors,
but they have infinite anisotropy energy $A(\m)$,
so we first need to modify them. 
We introduce a truncation parameter $L>0$ and define
$w_*^L\colon\C\to \C$ as
 \begin{align}\label{eq:w*L}
w_*^L(z)
&=\chi(|z|/ L)w_*(z) +(1-\chi(|z|/ L))i
\\
&
=i -  \chi(|z|/L) \frac{2i}{z+1}
\,,
\nonumber
\end{align}
where $w_*(z)=i(z-1)/(z+1)$ as in Lemma~\ref{l:param_mobius}
and $\chi$ is a smooth cut-off function satisfying
\begin{align*}
\mathbf 1_{r\leq 1}\leq\chi(r)\leq\mathbf 1_{r\leq 2},
\quad 0\geq \chi'\geq -2\,.
\nonumber
\end{align*} 
In analogy with the parametrization of Möbius maps given in Lemma~\ref{l:param_mobius},
this
truncated function $w_*^L$ can be used to define general truncated Möbius maps
\begin{align*}
\m_L^{\lbrace z_0,\rho,\phi,\alpha,\beta\rbrace} 
&
=
\mathfrak T_{z_0}
\mathfrak D_\rho
\mathfrak R_\phi 
\m_L^{[\alpha,\beta]},
\quad z_0\in\C,\, \rho >0,\,\phi,\sigma,\alpha,\beta\in\R\,,
\\
\m_L^{[\alpha,\beta]}
&
=R_{\mathbf e_1,2\beta}\Phi(w_*^L(e^{-i\alpha}z))\,.
\end{align*}
Recall that here $\Phi$ is the stereographic map defined in \eqref{eq:stereo}.
Note that the invariances of the Dirichlet energy ensure
\begin{align*}
&
\int_{\R^2} 
|\nabla \m_L^{\lbrace z_0,\rho,\phi,\alpha,\beta\rbrace}-\nabla \m^{\lbrace z_0,\rho,\phi,\alpha,\beta\rbrace}|^2\, dx
\\
&
=
\int_{\R^2} 
|\nabla [\Phi(w_*)-\Phi(w_*^L)]|^2\, dz
\\
&
=\int_{|z|\geq L}
|\nabla [\Phi(w_*)-\Phi(w_*^L)]|^2\, dz
\leq \frac{C}{L^2}\,,
\end{align*}
so for large enough $L$ these modified maps $\m_L$ must satisfy $Q(\m_L)=-1$,
and belong to the admissible set $\mathcal W_{-1}$ defined in \eqref{eq:Wq}.

If $\beta\neq 0$ modulo $\pi/2$, these maps have infinite anisotropy $A(\m)$,
and all values of $\beta =0$ modulo $\pi/2$ give the same energy, 
so we assume $\beta=0$.
The invariances \eqref{eq:Tx0Rphi} allow us to assume without loss of generality $z_0=0$ and $\phi=0$.
We are therefore left with three free parameters and denote
\begin{align*}
\m_{\alpha,\rho,L}(z)=\mathfrak D_\rho \m_L^{[\alpha,0]}(z)=\Phi\Big( w_*^L\Big(e^{-i\alpha}\frac{z}{\rho}
\Big)
\Big)\,.
\end{align*}
Using the invariances of each energy term,
 the properties
\begin{align*}
&
(\m_{\alpha,1,L})^2_3 =|\nabla m_{\alpha,1,L}|^2=0\quad\text{in }\R^2\setminus B_{2L}\,,
\\
&
(\m_{\alpha,1,L})^2_3  \leq \frac{C}{L^2}
\quad
\text{and }
|\nabla m_{\alpha,1,L}|^2
\leq \frac{C}{L^4}
\quad\text{in }B_{2L}\setminus B_L\,,
\end{align*}
and
 $w_*^L=w_*$ in $B_L$,
 we find
\begin{align*}
D(\m_{\alpha,\rho,L} )
&
=
D(\m_{0,1,L})
=\int_{B_L} \big|\nabla [\Phi(w_*) ]\big|^2\, dz
+\mathcal O(1/L^2),
\\
\widetilde H(\m_{\alpha,\rho,L} )
&
=
\rho \widetilde H(\m_{\alpha,1,L})
=
\rho\widetilde H(\m^{[\alpha,0]};B_L)
+ \mathcal O(\rho/L)\,,
\\
A(\m_{\alpha,\rho,L} )
&
=\rho^2 A(\m_{0,1,L} )
=\rho^2\int_{B_L}\Phi_3^2(w_*)\, dz   + \mathcal O(\rho^2)\,.
\end{align*}
The expressions involving $\Phi(w_*)$  and $\m^{[\alpha,0]}$ in the right-hand sides
can be calculated explicitly.
For the Dirichlet energy we know that the limit is $4\pi$ as $L\to\infty$,
and the decay $|\nabla[\Phi(w_*)]|^2\leq C/|z|^4$ as $|z|\to\infty$
 gives an error of order $1/L^2$.
For the DMI energy we use 
\eqref{eq:DMImalphabeta} and
\eqref{eq:DMImalphabetaR}.
And for the anisotropy term we have 
\begin{align*}
&
\int_{B_L}\Phi_3^2(w_*)\, dz 
=\int_{B_L} \frac{(|w_*|^2-1)^2}{(|w_*|^2+1)^2}\, dz
\\
&
=\int_{B_L}
\frac{(|z-1|^2-|z+1|)^2}{(|z-1|^2+|z+1|)^2}\, dz
=4\int_0^{2\pi}\int_0^L
 \frac{\cos^2\theta\, r^2}{(1+r^2)^2}\, rdr\, d\theta
 \\
 &
 =
 2\pi\int_0^{L^2}\frac{t}{(1+t)^2}\, dt
= 2\pi \ln(1+L^2) -\frac{2\pi L^2}{1+L^2}
\end{align*}
Thus we find
\begin{align*}
D(\m_{\alpha,\rho,L} )
&
=4\pi +\mathcal O(1/L^2),
\\
\widetilde H(\m_{\alpha,\rho,L} )
&
=
-4\pi \rho\, \cos(\alpha) + \mathcal O(\rho/ L)\,,
\\
A(\m_{\alpha,\rho,L} )
&
=4\pi \rho^2 \ln L + \mathcal O(\rho^2)\,,
\end{align*}
and deduce
\begin{align}
&
E_\sigma(
    \m_{\alpha,\rho,L})
\leq 4\pi +\mathcal{E}_\sigma(\alpha,\rho,L)
 + \mathcal O(\sigma^2\rho/ L + \sigma^2\rho^2),
 \label{eq:up_reduced}
\\
&
\text{where }
   \mathcal{E}_\sigma(\alpha,\rho,L)
    =\frac {C_1}{L^2} +4\pi\sigma^2
\left(
\rho^2\ln L - \rho\cos(\alpha)
\right)\,,
\nonumber
\end{align}
for some absolute constant $C_1 > 0$. 
Minimizing $\mathcal E_\sigma$  over $\alpha$ gives $\alpha=0$
and
\begin{align*}
\mathcal{E}_{\sigma}(0,\rho,L) =\frac{C_1}{L^2} +4\pi\sigma^2
\left(
\rho^2\ln L - \rho
\right).
\end{align*}
Minimizing  over $\rho$ gives $\rho_L=1/(2\ln L)$ and 
\begin{align*}
\mathcal{E}_{\sigma}(0,\rho_L,L) = \frac{C_1}{L^2} - \frac{\pi\sigma^2}{\ln L}\,.
\end{align*}
Minimizing  over $L$ leads to, at main order for $\sigma\to 0$,
\begin{align*}
L_\sigma
&
=\sqrt{\frac{2C_1}{\pi}}\frac{\ln(1/\sigma)}{\sigma}\,,
\\
\rho_\sigma
&
=\rho_{L_\sigma}
 =\frac{1}{2\ln(1/\sigma)} \Big(1+\mathcal O\Big(\frac{\ln\ln(1/\sigma)}{\ln(1/\sigma)}\Big)\Big)\,,
\end{align*}
and
\begin{align}
\mathcal{E}_{\sigma}(0,\rho_{\sigma},L_\sigma) =
 -\frac{\pi\sigma^2}{\ln\big(\sigma^{-1}\ln(1/\sigma)\big)}
+\mathcal O\left(\frac{\sigma^2}{\ln^2\sigma}\right)
\,. \nonumber
\end{align}
Plugging this into \eqref{eq:up_reduced},
we deduce the upper bound \eqref{eq:up}.
\end{proof}

\begin{rem}\label{r:truncation}
The constant $C_1$ 
in the above proof
depends on the truncation \eqref{eq:w*L} 
that we used to
transform Möbius maps into maps with finite anisotropy energy $A(\m)$.
In \cite{bernand2021arma}, 
this modification of Möbius maps is done in a much more refined way, 
in order
to obtain an optimal constant $C_1$.
However, we see in the above proof
that the precise value of $C_1$ 
does not affect the upper bound 
at main order when $\sigma\to 0$.
That refinement is therefore superfluous here,
and it seems to us that the results in \cite{bernand2021arma} could also be obtained without that refinement.
\end{rem}

\section{Energy lower bound}\label{s:low}

In this section we consider
a map
$\m$ 
in the space $\mathcal W_{-1}$ defined in \eqref{eq:Wq}, that is, 
$\m\in \mathcal H(\R^2;\mathbb S^2)$
 such that
$A(\m)<\infty$ and $Q(\m)=-1$, and prove a sharp lower bound on its energy,
following quite closely the strategy in \cite{bernand2021arma},
with some necessary adaptations.

The most important tool in that strategy
is a stability estimate for the Möbius group \eqref{eq:mobmaps} as minimizers of the Dirichlet energy $D(\m)$, proved in \cite[Theorem~2.4]{bernand2021arma}.
That theorem provides an absolute constant $c_*>0$
(which could be made explicit) such that,
 for any $\m\in H_c^1(\R^2;\mathbb S^2)$, there exists a Möbius map $\Psi\in\mathcal M$
satisfying
    \begin{align}\label{stability-hm}
\int_{\R^2}\big|\nabla (\m-\Psi )\big|^2dx\le c_* \Big(\int_{\R^2}|\nabla\m|^2dx-8\pi\Big).
    \end{align} 
Moreover, it is apparent from the alternative proofs of this result in \cite{hirsch2022,topping2023} that the map $\Psi$ can be chosen so that 
\begin{align*}
u = \m\circ\Psi^{-1}-\mathrm{id}_{\mathbb S^2}=(\m-\Psi)\circ\Psi^{-1}\,,
\end{align*} 
has zero average on $\mathbb S^2$. 
Applying the  Moser-Trudinger inequality  on $\mathbb S^2$
\cite{moser1971iumj} and changing variables, this implies
\begin{align*}
&
\int_{\R^2} |\nabla \Psi|^2 
\exp\bigg( \frac{|\m-\Psi|^2}{\int_{\R^2}|\nabla (\m-\Psi)|^2\,dx}\bigg)
\, dx
\nonumber
\\
&
=2\int_{\mathbb S^2}  \exp\bigg(
\frac{ |u|^2}
{\int_{\mathbb S^2} |\nabla u |^2
d\mathcal H^2}
\bigg)\, d\mathcal H^2 
\leq c_{\mathrm{MT}}\,,
\end{align*}
for some explicit absolute constant $c_{\mathrm{MT}}>0$.
Defining
\begin{align}\label{eq:L}
    & L=\left(\int_{\mathbb{R}^2}|\nabla\m|^2dx-8\pi\right)^{-1/2}\,,
\end{align}
and setting
\begin{align*}
\bv
=\m-\Psi\,,
\qquad
\Psi 
=\m^{\lbrace z_0,\rho,\phi,\alpha,\beta\rbrace}\,,
\end{align*}
for some
 $z_0\in\C$, $\rho>0$ and $\phi,\alpha,\beta\in\R$
 (according to the parametrization of $\mathcal M$ 
 provided by Lemma~\ref{l:param_mobius})
 we have therefore
\begin{align}
\label{eq:stabMob}
&
 \int_{\R^2}\left|\nabla\bv\right|^2\, dx \leq \frac {c_*}{L^2},
\\
\text{and }
&
\int_{\R^2}
\exp\bigg(
\frac{v_3^2}{\int_{\R^2}|\nabla v_3|^2}
\bigg)
|\nabla \Psi|^2
\, dx  
\leq 
c_{\mathrm{MT}}\,
.
\label{eq:mosertrudinger}
\end{align}
In the next two subsections we use these stability estimates to 
provide lower bounds on the anisotropy and DMI energy in terms of $\Psi$ and $L$.

\subsection{Lower bounds for the anisotropy term}\label{ss:lowA}

In this section we prove two lower bounds on the anisotropy term
\begin{align*}
A(\m)=\int_{\R^2}m_3^2\, dx\,.
\end{align*}
The first lower bound serves to show that the angle $\beta$ must be close to $0$ modulo $\pi/2$, 
which then makes the second lower bound quite sharp.
The proofs are natural modifications of the
 two lower bounds in \cite[Lemma~6.1 \& Lemma~6.4]{bernand2021arma}.

\begin{lem}\label{l:lowA}
There exist $L_0,C>0$ depending on $c_{\mathrm{MT}}$ and $c_*$ such that, if $L$ defined in \eqref{eq:L} satisfies $L\geq L_0$, then
\begin{align}
&
\frac{A(\m)}{\rho^2}
\geq
\frac{1}{C} \sin^2(2 \beta) L^2\,,
\label{eq:lowAsin}
\\
\text{and}\quad
&
\frac{A(\m)}{\rho^2}
\geq 
4\pi \cos^2(2\beta)\ln L -C\,,
\label{eq:lowAcos}
\end{align}
where $\rho>0$ and $\beta\in\R$ are such that \eqref{eq:stabMob} and \eqref{eq:mosertrudinger} 
are satisfied.
\end{lem}

\begin{rem}\label{r:lowA}
From \eqref{eq:lowAsin} we infer
\begin{align*}
\cos^2(2\beta)
=1-\sin^2(2\beta)\geq 1-\frac{C}{L^2}\frac{A(\m)}{\rho^2}\,.
\end{align*}
Plugging this into \eqref{eq:lowAcos}, we deduce
\begin{align*}
\bigg(1+4\pi C\frac{\ln L}{L^2}\bigg)
\frac{A(\m)}{\rho^2}
\geq 4\pi\ln L - C\,,
\end{align*}
and therefore
\begin{align}\label{eq:lowA}
\frac{A(\m)}{\rho^2}
\geq
4\pi\ln L - C\,,
\end{align}
for all $L\geq L_0$ and a possibly larger constant $C>0$.
\end{rem}

\begin{proof}[Proof of Lemma~\ref{l:lowA}]
Using the invariances, we assume without loss of generality that
$\rho=1$, $z_0=0$, $\phi=0$ and $\alpha=\pi/2$,
so by \eqref{eq:stabMob} we have
\begin{align*}
&
\m =\m^{[\pi/2,\beta]}+\bv,\qquad
\int_{\R^2}|\nabla\bv|^2\, dx \leq \frac{c_*}{L^2}\,.
\end{align*}
Recalling Remark~\ref{r:walphabetaId}, we have
\begin{align*}
\m^{[\pi/2,\beta]}=R_{\mathbf e_1,2\beta -\pi/2}\Phi\,,
\end{align*}
and we deduce that the third component of $\m$ is given by
\begin{align}\label{eq:m3Phiv3}
m_3
&
=
-\cos(2\beta)\Phi_2 +\sin(2\beta)\Phi_3 +v_3\,.
\end{align}
From this identity, the proofs of the first and second lower bound \eqref{eq:lowAsin} and \eqref{eq:lowAcos} follow different strategies.
For the first, we fix $R\geq 1$, square \eqref{eq:m3Phiv3},
use the elementary inequality $(a+b)^2\geq a^2/2 - 2b^2$
and integrate on $B_R$, which gives
\begin{align*}
\int_{B_R} m_3^2\, dx
&
\geq \frac 12 \int_{B_R}\big(\sin(2\beta)\Phi_3-\cos(2\beta)\Phi_2)^2\, dx 
\\
&
\quad
- 2\int_{B_R}v_3^2\, dx\,.
\end{align*}
Using polar coordinates $x=re^{i\theta}$ and the explicit expression \eqref{eq:stereo} of the stereographic map $\Phi$ we can explicitly calculate the first integral in the right-hand side,
\begin{align*}
&
\int_{B_R}\big(\sin(2\beta)\Phi_3-\cos(2\beta)\Phi_2)^2\, dx
\\
&
=2\pi\sin^2(2\beta)\int_0^R
\frac{(r^2-1)^2}{(r^2+1)^2}\, r dr
+4\pi\cos^2(2\beta)\int_0^R\!\frac{r^2}{(r^2+1)^2}\,r dr
\\
&
=\pi\sin^2(2\beta) \Big( R^2 -6\ln(1+R^2)+\frac{6}{1+R^2}\Big)
\\
&
\quad
+2\pi
\Big(\ln(1+R^2)+\frac{1}{1+R^2}\Big)
\\
&
\geq
\frac \pi 2 \sin^2(2\beta) R^2 + 2\pi\ln R\,,
\end{align*}
if $R\geq R_0$ for a large enough absolute constant $R_0\geq 1$.
Next we estimate the integral of $v_3^2$, 
relying as in \cite[Lemma~6.1]{bernand2021arma} on the Moser-Trudinger inequality \eqref{eq:mosertrudinger}.
First we apply the inequality
$xy\leq e^x +y\ln (y/e)$ 
with
$x=v_3^2/\int|\nabla v_3|^2 dx$ and $y=|\nabla \Phi|^{-2}$, 
to write
\begin{align*}
&
\frac{\int_{B_R}v_3^2\, dx}
{\int_{\R^2}|\nabla v_3|^2\, dx}
=
\int_{B_R}
\frac{v_3^2}{\int_{\R^2}|\nabla v_3|^2\, dx}
|\nabla\Phi|^{-2}  |\nabla\Phi|^2\, dx
\\
&
\leq
\int_{B_R} 
\exp
\bigg(
\frac{v_3^2}{\int_{\R^2}|\nabla v_3|^2\, dx}
\bigg)|\nabla\Phi|^2\, dx
+
\int_{B_R}
\ln\bigg(
\frac{1}{e\,|\nabla\Phi|^2}
\bigg)
\, dx.
\end{align*}
Noting that 
$|\nabla\Phi|^2=8/(1+|x|^2)^2=|\nabla \m^{[\pi/2,\beta]}|^2$ and 
recalling the Moser-Trudinger inequality \eqref{eq:mosertrudinger} where
$\Psi=\m^{[\pi/2,\beta]}$, we deduce
\begin{align*}
&
\frac{\int_{B_R}v_3^2\, dx}
{\int_{\R^2}|\nabla v_3|^2\, dx}
\leq
c_{\mathrm{MT}}
+16\pi R^2
\ln R\,,
\end{align*}
and combining this with the stability estimate \eqref{eq:stabMob}
gives
\begin{align}\label{eq:estimv3L2}
&
\int_{B_R}v_3^2\, dx
\leq
c_{\mathrm{MT}}\frac{c_*}{L^2}
+16\pi c_* \frac{ R^2}{L^2}
\ln R\,.
\end{align}
Gathering the above inequalities, we infer
\begin{align*}
\int_{\R^2} m_3^2\, dx
&
\geq 
\frac \pi 4 \sin^2(2\beta) R^2 + \pi\ln R
- 2
c_{\mathrm{MT}}\frac{c_*}{L^2}
\\
&
\quad
-16\pi c_* \frac{ R^2}{L^2}
\ln R\,.
\end{align*}
Choosing $R=L/(4 \sqrt{ 2c_*})$, this becomes
\begin{align*}
\int_{\R^2} m_3^2\, dx
&
\geq 
\frac{\pi}{2^{7} c_*} 
\sin^2(2\beta) L^2 + \frac\pi 2 \ln L
- 2
c_{\mathrm{MT}}\frac{c_*}{L^2} -\frac\pi 2 \ln (4\sqrt {2c_*})
\\
&
\geq \frac{\pi}{2^{7}  c_*} 
\sin^2(2\beta) L^2\,,
\end{align*}
provided $L\geq L_0$ for a large enough $L_0\geq 1$,
and proves the first lower bound \eqref{eq:lowAsin}.

The proof of the second lower bound \eqref{eq:lowAcos} follows the strategy of \cite[Lemma~6.4]{bernand2021arma} 
relying on the Fourier transform
\begin{align*}
\mathcal F\varphi(\xi) =
\int_{\R^2}e^{-ix\cdot\xi}\varphi(x)\, dx\,.
\end{align*}
Differentiating the identity \eqref{eq:m3Phiv3} and taking Fourier transforms we have, for $\ell=1,2$,
\begin{align*}
i\xi_\ell \mathcal F m_3
&
=-\cos(2\beta)\mathcal F(\partial_\ell \Phi_2) +\sin(2\beta)\mathcal F(\partial_\ell\Phi_3) +\mathcal F(\partial_\ell v_3)\,.
\end{align*}
Since $\Phi_3$ and $\Phi_2/x_2$ are radial and real-valued,
a direct  calculation using polar coordinates shows that
\begin{align*}
\re \mathcal F(\partial_\ell\Phi_3)=\im\mathcal F(\partial_\ell\Phi_2)=0\,,
\end{align*}
and we deduce
\begin{align*}
\frac{\xi_\ell}{|\xi|} \im \mathcal F m_3 = \cos(2\beta)\frac{\mathcal F(\partial_\ell\Phi_2)}{|\xi|} -\frac{\re \mathcal F(\partial_\ell v_3)}{|\xi|}\,.
\end{align*}
We fix $\mu\geq 0$, to be chosen later.
Applying the identity
\begin{align*}
&
\int_{\R^2}|f|^2\, d\xi
-
\int_{\R^2}
\frac{\mu |\xi|^2}{1+\mu |\xi|^2} |g|^2\, d\xi
+
\mu\int_{\R^2}|\xi|^2 |f-g|^2\, d\xi
\nonumber
\\
&
=\int_{\R^2}(1+\mu |\xi|^2)
\Big| f-
\frac{\mu |\xi|^2}{1+\mu |\xi|^2}g
\Big|^2\, d\xi
\geq 0
\,
,
\end{align*}
valid for any $f\in L^2(\R^2;(1+|\xi|^2)d\xi)$ and $g\in L^2(\R^2;|\xi|^2d\xi)$,
to $f=\xi_\ell\im \mathcal F m_3/|\xi| $ and $g=
\cos(2\beta)  \mathcal F(\partial_\ell\Phi_2)/|\xi| $,
and summing over $\ell=1,2$ we obtain
\begin{align*}
\int_{\R^2}(\im \mathcal F m_3)^2\, d\xi
&
\geq
\cos^2(2\beta)
\int_{\R^2}
\frac{\mu |\xi|^2}{1+\mu |\xi|^2}\frac{|\mathcal F(\nabla \Phi_2)|^2}{|\xi|^2}d\xi
\\
&
\quad
- \mu \int_{\R^2}|\mathcal F(\nabla v_3)|^2\, d\xi\,.
\end{align*}
Using Plancherel's identity
\begin{align*}
4\pi^2\int_{\R^2}|\varphi|^2\, dx  =\int_{\R^2}|\mathcal F\varphi|^2\, d\xi\,,
\end{align*}
this implies
\begin{align}\label{eq:lowAsin1}
\int_{\R^2}m_3^2\, dx
&
\geq 
\cos^2(2\beta)
\int_{\R^2}
\frac{\mu |\xi|^2}{1+\mu |\xi|^2}\frac{|\mathcal F(\nabla \Phi_2)|^2}{4\pi^2|\xi|^2}d\xi
\nonumber
\\
&\quad
- \mu \int_{\R^2}|\nabla v_3|^2\, dx\,.
\end{align}
As in \cite{bernand2021arma}, 
the first integral in the right-hand side can be explicitly calculated.
It is shown in \cite[Lemma~A.5]{bernand2021arma} that
\begin{align*}
\mathcal F(\nabla \Phi_2)
&
=-4\pi K_1(|\xi|)\xi_2 \frac{\xi}{|\xi|}\,,
\end{align*}
where $K_1$ is a modified Bessel function \cite[\S~9.6]{abramowitz1972handbook}.
Thus we have
\begin{align*}
&
\int_{\R^2}
\frac{\mu |\xi|^2}{1+\mu |\xi|^2}\frac{|\mathcal F(\nabla \Phi_2)|^2}{4\pi^2|\xi|^2}d\xi
=4 \pi \int_0^\infty \frac{\mu r^2}{1+\mu r^2} K_1(r)^2\, rdr\,,
\end{align*}
and, using the known asymptotics of Bessel functions \cite[\S~9.6.11]{abramowitz1972handbook},
\begin{align*}
K_1(r)=
\frac 1r +\mathcal O(r\ln r)
&\quad\text{as }r\to 0\,,
\end{align*}
 we infer
\begin{align*}
&
\int_0^\infty \frac{\mu r^2}{1+\mu r^2} K_1(r)^2\, rdr
\geq 
\int_0^1 \frac{\mu r^2}{1+\mu r^2} K_1(r)^2\, rdr
\\
&\geq
\int_0^1\frac{\mu r}{1+\mu r^2}\, dr
- c_1
\int_0^1 r |\ln r| \, dr\,
\end{align*}
for some absolute constant $c_1>0$,
and therefore
\begin{align*}
&
\int_{\R^2}
\frac{\mu |\xi|^2}{1+\mu |\xi|^2}\frac{|\mathcal F(\nabla \Phi_2)|^2}{4\pi^2|\xi|^2}d\xi
\\
&
\geq
4\pi\int_0^1\frac{\mu r}{1+\mu r^2}\, dr -c_1
=2\pi\ln (1+\mu)-c_1\,.
\end{align*}
Coming back to the lower bound \eqref{eq:lowAsin1} on $\int m_3^2\, dx$
and using also the stability estimate \eqref{eq:stabMob}, this gives
\begin{align*}
\int_{\R^2}m_3^2\, dx
\geq 
2\pi \cos^2(2\beta)\ln (1+\mu) -c_1 -\frac{c_*\mu}{L^2}\,.
\end{align*}
Choosing $\mu=c_1 L^2/c_*$ we infer
\begin{align*}
\int_{\R^2}m_3^2\, dx
\geq 
4\pi \cos^2(2\beta) \ln L -2c_1 - |\ln(c_1/c_*)|\,,
\end{align*}
which proves the second lower bound \eqref{eq:lowAsin}.
\end{proof}

\subsection{Lower bound for the DMI term}\label{ss:lowH}

In this section we prove a lower bound on the DMI term
\begin{align*}
\widetilde H(\m)=
2\int_{\R^2}m_3(\partial_1 m_2-\partial_2 m_1)\, dx\,.
\end{align*}
Similar to Step~2 of \cite[Lemma~6.5]{bernand2021arma},
that lower bound is in term of the DMI energy of the Möbius map $\Psi$ 
and a small error term. 
Here the error term is not as good as in \cite{bernand2021arma},
this is due to the different form of our DMI term,
and the fact that the components of the stereographic 
map $\Phi$ given by \eqref{eq:stereo} have different integrability properties: 
$(\Phi_3+1)^2$ is integrable,  but $\Phi_1^2$ and $\Phi_2^2$ are not.
This forces us to use a slightly more involved argument to control the error term.

\begin{lem}\label{l:lowH}
There exist $L_0,C>0$ depending on $c_{\mathrm{MT}}$ and $c_*$ such that, if $L$ defined in \eqref{eq:L} satisfies $L\geq L_0$, then
\begin{align*}
\frac{\widetilde H(\m)}{\rho}
&
\geq
-4\pi\cos(\alpha) \cos^2(2\beta)
- C\sqrt{\frac{\ln L}{L}} - \frac{C}{\sqrt L}\sqrt{\frac {A(\m)}{\rho^2}}\,,
\end{align*}
where $\rho>0$ and $\alpha,\beta\in\R$ are such that \eqref{eq:stabMob} and \eqref{eq:mosertrudinger} are satisfied.
\end{lem}
\begin{proof}[Proof of Lemma~\ref{l:lowH}]
Using the invariances, we assume without loss of generality that
$\rho=1$, $z_0=0$ and $\phi=0$. Taking
Remark~\ref{r:walphabetaId} into account,
we are therefore left with
\begin{align*}
\m =\Psi +\bv,
\quad
 \Psi &=\m^{[\alpha,\beta]}=R_{\mathbf e_1,\beta-\pi/2}\widetilde\Phi\,,
\quad
\widetilde\Phi(z)
=\Phi(e^{-i\alpha}iz)\,,
\end{align*}
where $\Phi$ is the stereographic map defined in \eqref{eq:stereo},
and $\bv$ satisfies \eqref{eq:stabMob} and \eqref{eq:mosertrudinger}.
The integrand of $\widetilde H(\m)$ 
 satisfies the identities
\begin{align*}
&
m_3(\partial_1m_2-\partial_2 m_1)
-m_3 (\partial_1 v_2-\partial_2 v_1)
\nonumber
\\
&
=\Psi_3(\partial_1\Psi_2-\partial_2\Psi_1)
+
v_3(\partial_1\Psi_2 -\partial_2\Psi_1)
\nonumber
\\
&
=m_3 (\partial_1\Psi_2 -\partial_2\Psi_1)\,,
\end{align*}
from which we infer
\begin{align*}
&
\widetilde H(\m)
 -2\int_{\R^2}m_3 (\partial_1 v_2-\partial_2 v_1)\, dx
\\
&
=\widetilde H(\Psi;B_{\sqrt L}) +2\int_{B_{\sqrt L}}v_3(\partial_1\Psi_2-\partial_2\Psi_1)\, dx
\\
&\quad
+2\int_{|x|\geq \sqrt L} m_3(\partial_1\Psi_2-\partial_2\Psi_1)\, dx\,.
\end{align*}
This implies
\begin{align*}
&
\frac12 \left( \widetilde H(\m)
-
\widetilde H(\Psi;B_{\sqrt L}) \right)
\\
&
=
\int_{\R^2}m_3 (\partial_1 v_2-\partial_2 v_1)\, dx
+
\int_{B_{\sqrt L}}v_3(\partial_1\Psi_2-\partial_2\Psi_1)\, dx
\\
&\quad
+\int_{|x|\geq \sqrt L} m_3(\partial_1\Psi_2-\partial_2\Psi_1)\, dx\,,
\end{align*}
and, using the Cauchy-Schwarz inequality,
\begin{align*}
&
\frac 1 8 \Big(\widetilde H(\m)
-
\widetilde H(\Psi;B_{\sqrt L})\Big)^2
\\
&
\leq
A(\m) \int_{\R^2}|\nabla \bv|^2\,dx
+\int_{B_{\sqrt L}}
v_3^2\, dx
\int_{\R^2}|\nabla\Psi|^2\,dx
\\
&
\quad
+
A(\m)\int_{|x|\geq\sqrt L}|\nabla\Psi|^2\, dx\,.
\end{align*}
Using the fact that
$|\nabla\Psi|^2=|\nabla\Phi|^2=8/(1+|x|^2)^2$,
 the stability estimate \eqref{eq:stabMob} on $\int|\nabla\bv|^2dx$,
and the estimate \eqref{eq:estimv3L2} on integrals of $v_3^2$, this implies
\begin{align*}
&
\frac 14\Big(\widetilde H(\m)
-
\widetilde H(\Psi;B_{\sqrt L})\Big)^2
\\
&
\leq
A(\m) \frac{c_*+8\pi}{L}
+8\pi\Big( 
c_{\mathrm{MT}}\frac{c_*}{L} + 8\pi c_* \frac{\ln L}{L}
\Big)
\end{align*}
Finally, recalling the explicit expression 
\eqref{eq:DMImalphabeta} of $\widetilde H(\Psi)$ 
and its error \eqref{eq:DMImalphabetaR} from $\widetilde H(\Psi;B_{\sqrt L})$,
we have
\begin{align*}
\widetilde H(\Psi;B_{\sqrt L})
&
\geq -4\pi\cos(\alpha) \cos^2(2\beta)-\frac{C}{L} \,,
\end{align*}
and combining this with the previous estimate 
gives the conclusion.
\end{proof}

\subsection{Lower bound for the full energy}\label{ss:lowE}

In this section we combine the lower bounds of Lemma~\ref{l:lowA} and Lemma~\ref{l:lowH}
with elementary 
calculations 
to deduce a sharp
energy lower bound
and characterize the case of near equality.

\begin{prop}\label{p:lowE}
There exists $C_0>0$ 
and $\sigma_0\in (0,1/4]$
 depending explicitly on $c_{\mathrm{MT}}$ and $c_*$ such that, for any map $\m\in \mathcal W_{-1}$ defined in
\eqref{eq:Wq}
and 
$0<\sigma <\sigma_0$, 
the energy $E_\sigma(\m)$ defined in \eqref{eq:Esigma} is bounded below by
\begin{align}\label{eq:lowE}
E_\sigma(\m)\geq 4\pi 
-\frac{\pi\sigma^2}{\ln\big(\sigma^{-1}\ln(1/\sigma) \big)} 
- C_0\frac{\sigma^2}{\ln^2\sigma}\,.
\end{align}
Moreover, if  $\m=\m_\sigma$ saturates that lower bound, in the sense that
\begin{align}\label{eq:lowsatmsigma}
E_\sigma(\m_\sigma)\leq 4\pi 
-\frac{\pi\sigma^2}{\ln\big(\sigma^{-1}\ln(1/\sigma) \big)} 
+ K\frac{\sigma^2}{\ln^2\sigma}\,,
\end{align}
for some $K \geq C_0$, then there exists a constant $C(K)>0$ depending explicitly on $K$  and a Möbius map
$\Psi = m^{\lbrace z_0,\rho,\phi,\alpha,\beta \rbrace}$
 as in \eqref{eq:param_mobius}
such that
\begin{align*}
\frac{1}{C(K)}
\frac{\sigma^2}{\ln^2 \sigma}
\leq
\int_{\R^2}|\nabla \m_\sigma -\nabla\Psi|^2\, dx \leq C(K)
\frac{\sigma^2}{\ln^2\sigma}\,,
\end{align*}
and the parameters $\rho>0$, $\alpha,\beta\in\R$ satisfy
\begin{align*}
&
\bigg|
\rho-\frac{1}{2\ln(1/\sigma)}\bigg|
\leq \frac{C(K)}{\ln^{\frac 32}(1/\sigma)}\,,
\quad
|\alpha |
\leq
\frac{C(K)}{\sqrt{\ln(1/\sigma)}}\,,
\\
&
\text{and}
\quad
|\beta| \leq \frac{\sigma}{\sqrt{\ln (1/\sigma)}}\,.
\end{align*}
\end{prop}

\begin{proof}[Proof of Proposition~\ref{p:lowE}]
We first note, for $0<\sigma<1/4$, the basic lower bound
\begin{align}\label{lower}
    E_{\sigma}(\m)
    &
    \geq
    \frac {\sigma^2}{2}A(\m) + (1-8\sigma^2) D(\m)
    \nonumber
    \\
    &
    \geq \frac {\sigma^2}{2}A(\m) + (1-8\sigma^2) 4\pi|Q(\m)|
    \,,
\end{align}
where the last inequality follows from \eqref{eq:DQ} and the first inequality from
\begin{align*}
\big|\widetilde H(\m)\big|
   &
   =2\bigg|\int_{\R^2}m_3(\partial_1 m_2-\partial_2 m_1)\, dx 
   \bigg|
   \\
   &
\leq \frac 12 \int_{\R^2}|m_3|^2\,dx+ 2 \int_{\R^2}(\partial_1m_2-\partial_2m_1)^2\,dx
   \\
&
\leq \frac 12 A(\m) +8 D(\m)\,.
\end{align*}
We may 
assume without loss of generality that the map $\m$ satisfies
\begin{align*}
E_\sigma(\m)\leq 4\pi\,,
\end{align*}
since otherwise \eqref{eq:lowE} is automatically true.
From the basic lower bound \eqref{lower} and the definition
\eqref{eq:L} of $L$, this implies, for $0<\sigma<\sigma_0\leq 1/4$,
\begin{align}\label{eq:basicAL}
A(\m)\leq 64\pi\,\quad
\text{and}
\quad
L\geq\frac{1}{4\sigma}\,.
\end{align}
If $\sigma_0$ is small enough, then  $L\geq L_0$, for $L_0$ as in 
Lemma~\ref{l:lowA} and Lemma~\ref{l:lowH}.
During the rest of the proof we will assume that $L_0$ is as large as we need, 
and will denote by $C$ a generic constant which may change from line to line, 
but whose explicit dependence on $c_*$ and $c_{\mathrm{MT}}$ can be kept explicit track of.

Combining the definitions \eqref{eq:Esigma} of the energy $E_\sigma$ and \eqref{eq:L} of $L$
with the lower bound
 for the DMI term $\widetilde H(\m)$, we find
\begin{align*}
E_\sigma(\m)
&
\geq
4\pi +\frac{1}{2L^2} +\sigma^2 A(\m)
 -4\pi\sigma^2\rho\cos(\alpha)\cos^2(2 \beta)
 \\
&
- C\sigma^2 \frac{\sqrt{A(\m)}+\rho\sqrt{\ln L}}{\sqrt L} \,,
\end{align*}
and, using 
\[2\rho\frac{\sqrt{\ln L}}{\sqrt{L}}\le \rho^2+\frac{\ln L}{L}\le \rho^2+\frac{1}{\sqrt{L}}\]
and \eqref{eq:basicAL} to estimate the last term,
\begin{align}
&
E_\sigma(\m)
-
4\pi 
+C\sigma^2 (\rho^2 + \sqrt\sigma )
\nonumber
\\
&
\geq
\frac{1}{2L^2} +\sigma^2 A(\m)
 -4\pi\sigma^2\rho\cos(\alpha)\cos^2(2\beta)
  \,.
 \label{eq:lowE0}
\end{align}
Plugging in the lower bound \eqref{eq:lowA}  of Remark~\ref{r:lowA} for the anisotropy term $A(\m)$
we find
\begin{align}\label{eq:lowE1}
&
E_\sigma(\m)
-
4\pi 
+C\sigma^2 (\rho^2 +\sqrt\sigma )
\nonumber
\\
&
\geq
\frac{1}{2L^2} +4\pi\sigma^2\rho^2\ln L -4\pi\sigma^2\rho
 +4\pi\sigma^2\rho\,(1-\cos\alpha)
 \nonumber
 \\
 &
 \quad
 +4\pi\sigma^2\rho\cos(\alpha)\sin^2(2\beta)
 \nonumber
\\
&
\geq
\frac{1}{2L^2} +4\pi\sigma^2\rho^2\ln L -4\pi\sigma^2\rho
 +4\pi\sigma^2\rho\,(1-\cos\alpha)
\,.
\end{align}
Using that the last term is nonnegative, that $E_\sigma(m)\leq 4\pi$,
and that
 $4\pi\sigma^2\rho\leq 2\pi\sigma^2\rho^2+4\pi\sigma^2$,
we deduce in particular that
\begin{align*}
4\pi\sigma^2\rho^2\ln L 
\leq C\sigma^2\rho^2 +C\sigma^{5/2}+4\pi\sigma^2\,.
\end{align*} 
If $L_0$ is large enough and $\sigma_0$ is small enough we can absorb the first term 
of the right-hand side into the first term of the left-hand side, and the second term of the right-hand side into the last, to deduce
\begin{align*}
2\pi\sigma^2\rho^2\ln L \leq 8\pi\sigma^2\,,
\end{align*}
and therefore, recalling that $L\geq 1/4\sigma$,
\begin{align}\label{eq:boundrho1}
\rho\leq  \frac{8}{\ln(1/\sigma)}\,.
\end{align}
Using this to estimate the third term in the first line of \eqref{eq:lowE1}, we obtain
 \begin{align}\label{eq:lowE2}
&
E_\sigma(\m)-4\pi 
+
C\frac{\sigma^{2}}{\ln^2\sigma}
\nonumber
\\
&
\geq
f_\sigma\big(\rho,\frac 1L\big) +4\pi\sigma^2\rho\,(1-\cos\alpha)
 \,,
\\
&
\text{where }
f_\sigma(\rho,t)
=\frac{t^2}{2}
-4\pi\sigma^2\rho^2\ln t -4\pi\sigma^2\rho\,.
\nonumber
\end{align}
The function $t\mapsto f_\sigma(\rho,t)$ is convex, with zero derivative at  $t_*(\rho)=2\sqrt\pi\sigma\rho$,
and we have the identity
\begin{align}\label{eq:fsigma-min}
f_\sigma(\rho,t)
-f_\sigma(\rho,t_*(\rho))
&
=\frac 12 (t-t_*(\rho))^2
\nonumber
\\
&\quad
+4\pi\sigma^2\rho^2g_*\big(t_*(\rho)/t\big),
\\
\text{where }
g_*(x)
&
=\ln x -1 +\frac 1x \geq 0\qquad\forall x>0\,.
\nonumber
\end{align}
Moreover, its minimal value at $t_*$ is given by
\begin{align*}
f_\sigma\big(\rho,t_*(\rho)\big)
&
=2\pi\sigma^2
\bigg(
\rho^2
+2\rho^2\ln\Big(\frac{1}{2\sqrt\pi\sigma\rho}\Big) -2\rho
\bigg)
\\
&
\geq
4\pi\sigma^2\Big(\rho^2\ln\Big(\frac{1}{2\sqrt\pi\sigma\rho}\Big)-\rho\Big)
\\
&
=-\frac{\pi\sigma^2}{\ln\big(\sigma^{-1}\ln(1/\sigma)\big)}
\\
&
\quad
+
4\pi\sigma^2\ln\Big(\frac 1{2\sqrt\pi\sigma\rho}\Big)
\Big(\rho
-\frac{1}{2\ln\big(1/(2\sqrt\pi\sigma\rho)\big)}
\Big)^2
\\
&
\quad
+
\frac{\pi\sigma^2\ln\big(1/(2\sqrt\pi\rho\ln(1/\sigma))\big)}{\ln\big(\sigma^{-1}\ln(1/\sigma)\big)
\ln\big(1/(2\sqrt\pi\sigma\rho)\big)}
\,.
\end{align*}
If $\sigma_0$ is small enough, then thanks to \eqref{eq:boundrho1} we have $2\sqrt\pi\rho\leq 1$ and the denominator of the last term is $\geq \ln^2\sigma$.
Combining this
 with 
 \eqref{eq:fsigma-min}
 and \eqref{eq:boundrho1},
 we deduce
\begin{align*}
f_\sigma (\rho,t)
&
\geq 
- \frac{\pi\sigma^2}{\ln\big(\sigma^{-1}\ln(1/\sigma)\big)} 
\\
&
\quad
+
4\pi\sigma^2\ln\Big(\frac{\ln(1/\sigma)}{\sigma}\Big)
\bigg(\rho
-\frac{1}{2\ln\big(1/(2\sqrt\pi\sigma\rho)\big)}
\bigg)^2
\\
&
\quad 
+4\pi\sigma^2\rho^2g_*\Big(\frac{2\sqrt\pi\sigma\rho}{t}\Big)
\\
&
\quad
+\frac{\pi\sigma^2 
\ln\big(4/(\rho\ln(1/\sigma))\big)
}
{
\ln^2\sigma
}
- C\frac{\sigma^2}{\ln^2\sigma}\,.
\end{align*}
Plugging  this inequality into the lower bound \eqref{eq:lowE2} 
we obtain
 \begin{align}\label{eq:lowE3}
&
\frac{E_\sigma(\m)-4\pi }{\sigma^2}
+
\frac{\pi}{\ln\big(\sigma^{-1}\ln(1/\sigma) \big)}
+
\frac{C_0}{\ln^2\sigma}
\nonumber
\\
&
\geq
4\pi \ln\Big(\frac{\ln(1/\sigma)}{\sigma}\Big)
\bigg(\rho
-\frac{1}{2\ln\big(1/(2\sqrt\pi\sigma\rho)\big)}
\bigg)^2
\nonumber
\\
&\quad
+\frac{\pi
\ln\big(4/(\rho\ln(1/\sigma))\big)
}
{
\ln^2\sigma
}
+4\pi\rho^2g_*\big(2 L \sqrt\pi\sigma\rho \big)
\nonumber
\\
&
\quad
+4\pi\rho\,(1-\cos\alpha)
\,,
\end{align}
for some absolute constant $C_0>0$ depending explicitly on $c_*$ and $c_{\mathrm{MT}}$.
Since all terms in the last three lines of \eqref{eq:lowE3} are nonnegative, 
this implies the lower bound \eqref{eq:lowE}.

\medskip

Now we assume that
the map $\m=\m_\sigma$ saturates that lower bound, that is, it satisfies the upper bound \eqref{eq:lowsatmsigma} 
for some $K\geq C_0$.
In the rest of the proof, we will be denote by $C(K)$ a generic positive constant depending on $K$ and the previous absolute constants, whose value may change from line to line.

Under the assumption \eqref{eq:lowsatmsigma},
 all the nonnegative terms in the right-hand side of \eqref{eq:lowE3} are bounded by $2K/\ln^2\sigma$.
For the second term, this implies
\begin{align*}
\rho
 \geq \frac{4e^{-2K/\pi}}{\ln(1/\sigma)}\,,
\end{align*}
and we deduce
\begin{align*}
&
\bigg|
\frac{1}{2\ln\big(1/(2\sqrt\pi\sigma\rho)\big)}
-\frac{1}{2\ln(1/\sigma)}
\bigg|
=
\frac
{\ln\big(1/(2\sqrt\pi\rho)\big)}
{2\ln(1/\sigma)\ln\big(1/(2\sqrt\pi\sigma\rho)\big)}
\\
&
\leq \frac{\ln\ln(1/\sigma)+CK}{2\ln^2\sigma}\,.
\end{align*}
This and the fact that the first term
in the right-hand side of \eqref{eq:lowE3} is $\leq 2K/\ln^2\sigma$
imply
\begin{align}\label{eq:asymptrho}
\bigg|\rho
-\frac{1}{2\ln(1/\sigma)}
\bigg|
\leq \frac{C(K)}{\ln^{\frac 32}(1/\sigma)}\,,
\end{align}
which is the claimed estimate on $\rho$.

Next we use that
the third term
in the right-hand side of \eqref{eq:lowE3} is $\leq 2K/\ln^2\sigma$ and that $C(K)\rho\geq \ln(1/\sigma)$, to obtain
\begin{align*}
g_*(2L\sqrt\pi\sigma\rho)\leq C(K),
\end{align*}
where we recall that $g_*(x)=\ln x -1+1/x$.
Since $g_*(x)\to +\infty$ as $x\to 0$ and $x\to\infty$, 
we infer
\begin{align*}
\frac{1}{C(K)}\leq L\sigma\rho \leq C(K),
\end{align*}
and using \eqref{eq:asymptrho}  this gives
\begin{align}\label{eq:boundsL}
\frac{1}{C(K)} \frac{\ln(1/\sigma)}{\sigma}
 \leq L \leq C(K)
 \frac{\ln(1/\sigma)}{\sigma}\,.
\end{align}
The fact that the fourth
term in the right-hand side of \eqref{eq:lowE3} is $\leq 2K/\ln^2\sigma$,
implies
\begin{align*}
\frac 1C \dist^2(\alpha,2\pi\Z) \leq 1-\cos\alpha \leq \frac{C(K)}{\ln(1/\sigma)}\,.
\end{align*}
Since $\alpha$ can be chosen in $[-\pi,\pi]$, this implies the claimed estimate on $\alpha$.

Finally, 
plugging the 
assumption
 \eqref{eq:lowsatmsigma} 
 and the estimate \eqref{eq:asymptrho} of $\rho$ 
 back into the inequality \eqref{eq:lowE0}
implies
\begin{align*}
A(\m)
&
\leq \frac{C(K)}{\ln(1/\sigma)}\,.
\end{align*}
Combining this with the first lower bound \eqref{eq:lowAsin} on $A(\m)$
and with the estimates \eqref{eq:asymptrho} and \eqref{eq:boundsL} of $\rho$ and $L$ gives
\begin{align*}
\sin^2(2\beta) 
&
\leq \frac{ C(K)\sigma^2}{\ln(1/\sigma)}\,,
\end{align*}
and therefore
\begin{align*}
\dist^2\Big(\beta,\frac\pi 2\Z\Big)
\leq \frac{\pi^2}{4} \sin^2(2\beta) \leq C(K)\frac{\sigma^2}{\ln(1/\sigma)}\,.
\end{align*}
Taking into account Remark~\ref{r:betaw*},
after possibly redefining the angle $\phi$ this implies the claimed estimate on $\beta$.

Recalling the definition \eqref{eq:L} of $L$, the above bounds \eqref{eq:boundsL}
turn into
\begin{align*}
\frac{1}{C(K)}
\frac{\sigma^2}{\ln^2 \sigma}
\leq
\int_{\R^2}|\nabla \m_\sigma|^2\, dx -8\pi \leq C(K)
\frac{\sigma^2}{\ln^2 \sigma }\,.
\end{align*}
Combining this with the stability estimate \eqref{stability-hm}
and  the classical identity
\begin{align*}
\int_{\R^2}|\nabla \m_\sigma|^2\, dx -8\pi
&
=
\int_{\R^2}|\nabla \m_\sigma -\nabla\Psi|^2\,dx
\\
&
\quad
-
\int_{\R^2}|\nabla\Psi|^2|\m_\sigma -\Psi|^2\,dx
\\
&
\leq\int_{\R^2}|\nabla \m_\sigma -\nabla\Psi|^2\,dx\,, 
\end{align*}
which follows from the harmonic map equation $-\Delta\Psi=|\nabla\Psi|^2\Psi$  and the identity
$|\m_\sigma -\Psi|^2 =2(\Psi-\m_\sigma)\cdot\Psi$,
we deduce
\begin{align*}
\frac{1}{C(K)}
\frac{\sigma^2}{\ln^2 \sigma}
\leq
\int_{\R^2}|\nabla \m_\sigma -\nabla\Psi|^2\,dx
 \leq C(K)
\frac{\sigma^2}{\ln^2 \sigma }\,,
\end{align*}
thus concluding the proof of Proposition~\ref{p:lowE}.
\end{proof}

\section{Existence of minimizers and proof of Theorem~\ref{t:main}}\label{s:exist}

In this section we rely on the characterization of near-minimizers in Proposition~\ref{p:lowE} 
to show that the infimum of $E_\sigma$ on $\mathcal W_{-1}$ is attained, provided $\sigma$ is small enough. 
This follows the
 standard concentration-compactness
strategy applied also in \cite{esteban1986commun,
lin2004cpam,melcher2014proc,
doring2017cvpde},
but we need a different argument to rule out ``vanishing''.

\begin{prop}\label{p:existence}
There exists an absolute constant $\sigma_0>0$, depending explicitly on $c_{\mathrm{MT}}$ and $c_*$, such that
\begin{align*}
\inf_{\mathcal W_{-1}} E_\sigma = \min_{\mathcal W_{-1}} E_\sigma\,,
\end{align*}
for $0<\sigma<\sigma_0$.
\end{prop}

Thanks to this existence result we may now prove our main theorem.

\begin{proof}[Proof of Theorem~\ref{t:main}]
The existence of a minimizer $\m_\sigma\in\mathcal W_{-1}$ is provided by Proposition~\ref{p:existence}.
The upper bound \eqref{eq:up} in the energy expansion is provided by the construction in \S~\ref{s:up}.
The lower bound and the description of the minimizer in terms of Möbius maps is provided by Proposition~\ref{p:lowE}, 
taking into account the parametrization of Möbius maps provided by Lemma~\ref{l:param_mobius}.
\end{proof}

Finally we prove that $E_\sigma$ attains its infimum on $\mathcal W_{-1}$.

\begin{proof}[Proof of Proposition~\ref{p:existence}]
The proof is divided in three steps: some basic observations on minimizing sequences, 
the conclusion under a tightness assumption, 
and finally a proof of  that tightness assumption.
All these steps follow well-known arguments,
but  
we provide a fair amount of details to convince the reader that they do adapt to our case.

\medskip
\noindent\textit{Step 1: Basic observations.}
\medskip

The first observation is that the energy can be 
rewritten as
\begin{align}
E_\sigma(\m)
&
=\int_{\R^2} e_\sigma(\m)\, dx,
\label{eq:Esigmaconvex}
\\
e_\sigma(\m)
&
=\frac 12 (\partial_2 m_1-2\sigma^2 m_3)^2
+\frac 12 (\partial_1 m_2 +2\sigma^2 m_3)^2
\nonumber
\\
&\quad 
+\frac 12 (\partial_1 m_1)^2 +\frac 12 (\partial_2 m_2)^2
+
\sigma^2(1-4\sigma^2)m_3^2\,.
\nonumber
\end{align}
For $0<\sigma<1/2$, this identity implies
\begin{align*}
\int_{\R^2} m_3^2\, dx  &
\leq \frac{1}{\sigma^2(1-4\sigma^2)} E_\sigma(\m)\,,
\\
\text{and }
\int_{\R^2}|\nabla \m|^2\, dx 
&
\leq 4 E_\sigma(\m) + 8\sigma^4\int_{\R^2} m_3^2\, dx\,,
\end{align*}
so any minimizing sequence $\m^{(k)}\in\mathcal W_{-1}$ satisfying
\begin{align*}
E_\sigma (\m^{(k)}) \to \inf_{\mathcal W_{-1}} E_\sigma\,,
\end{align*}
admits a subsequence, still denoted $\m^{(k)}$, such that
\begin{align*}
\nabla \m^{(k)}\rightharpoonup \nabla\m
\quad
\text{ and } 
\quad
m_3^{(k)}\rightharpoonup m_3\qquad\text{weakly in }L^2(\R^2)\,,
\end{align*}
for some $\m\in\mathcal W$, where the space $\mathcal W$ is defined in \eqref{eq:W}.
Moreover, the identity \eqref{eq:Esigmaconvex}
also implies that $E_\sigma$ is a convex function of $\m$,
and therefore satisfies the lower semicontinuity property
\begin{align}\label{eq:lsc}
&E_\sigma(\m)\leq \liminf E_\sigma(\m^{(k)})=\inf_{\mathcal W_{-1}}E_\sigma\,.
\end{align}
In order to conclude that $\m$ minimizes $E_\sigma$ in $\mathcal W_{-1}$, it remains to show that the weak limit $\m$ actually belongs to $\mathcal W_{-1}$.
This is not directly obvious because the topological degree $Q(\m)$ defined in \eqref{eq:Q} is not continuous with respect to that weak convergence.

Note that, if $0<\sigma<\sigma_0$ for a small enough $\sigma_0$,
 the upper bound \eqref{eq:up}
together with the basic lower bound
\eqref{lower} 
imply
\begin{align}\label{eq:highdeg}
\inf_{\mathcal W_{-1}} E_\sigma < \inf_{\mathcal W_q} E_\sigma\qquad\text{ if }|q|\geq 2\,.
\end{align}
From this and \eqref{eq:lsc} it follows that $Q(\m)\in\lbrace 0,\pm 1\rbrace$.
In order to conclude that $Q(\m)=-1$,
we therefore only need to show that $Q(\m)<0$.

\medskip
\noindent
\textit{Step 2. Conclusion under a tightness condition.}
\medskip

As in \cite{melcher2014proc}
the proof that $Q(\m)<0$ relies on a concentration-compactness argument.
That argument shows that,
along a non-relabeled subsequence and
modulo translations $\mathfrak T_{z_k}$ 
which keep the energy invariant \eqref{eq:Tx0Rphi},
 the $L^1$ sequence
\begin{align}
&f_k =|\nabla \m^{(k)}|^2 + (m_3^{(k)})^2
\quad
\text{ is uniformly tight, that is,}
\nonumber
\\
&
\sup_{k\geq 1} \int_{|x|\geq R} f_k\, dx \longrightarrow 0
\qquad\text{as }R\to + \infty \,.
\label{eq:tight}
\end{align}
That tightness, 
together with the strong convergence of $\m^{(k)}$ in $L^2(B_R)$ for any $R>0$ ensured by Rellich-Kondratchov's theorem,
 implies, as in \cite[Lemma~4.1]{melcher2014proc},
 \begin{align*}
 A(\m^{(k)}) +\widetilde H(\m^{(k)}) \longrightarrow  A(\m) +\widetilde H(\m)\,.
 \end{align*}
Moreover, 
classical arguments, 
see e.g.  \cite[Theorem~1.6]{struwe2008varmeth},
show that the functional
 $D(\m)+4\pi Q(\m)$ is lower semicontinuous
on $H^1_c(\R^2;\mathbb S^2)$ 
  with respect to the weak convergence $\nabla \m^{(k)}\rightharpoonup \nabla \m$ in $L^2(\R^2)$, because it can be written as 
\begin{align*}
D(\m)+4\pi Q(\m)
&
=\int_{\R^2} W(\nabla\m,\m)\, dx \,,\\
W(\nabla\m,\m)
&=\frac 12|\nabla \m|^2 -\m\cdot\partial_1\m\times\partial_2\m \geq 0,
\end{align*}
and $W(\m,\cdot)$ is a nonnegative, hence convex, quadratic form.
Since $E_\sigma =D +\sigma^2 (A+\widetilde H)$,
we infer the lower semicontinuity property
\begin{align*}
E_\sigma(\m) +4\pi Q(\m) \leq \liminf_{k\to\infty} \big(E_\sigma(\m^{(k)})+4\pi Q(\m^{(k)}) \big)\,.
\end{align*}
Recalling that $\m^{(k)}\in\mathcal W_{-1}$ is a minimizing sequence, we deduce 
\begin{align*}
E_\sigma(\m) + 4\pi Q(\m)  \leq \inf_{\mathcal W_{-1}}E_\sigma -4\pi \,.
\end{align*}
If $0<\sigma<\sigma_0$ for a small enough $\sigma_0$, the right-hand side is negative due to the upper bound  \eqref{eq:up}, and $E_\sigma(\m)\geq 0$ due to \eqref{lower},
so this implies 
$Q(\m) <0$, hence $Q(\m)=-1$ and $\m\in\mathcal W_{-1}$.
Together with \eqref{eq:lsc} 
this concludes the proof that the infimum of $E_\sigma$ on $\mathcal W_{-1}$ is attained,
provided we show the tightness property \eqref{eq:tight}.

\medskip
\noindent
\textit{Step 3. Proof of the tightness property.}
\medskip

As in \cite{melcher2014proc}, the tightness property \eqref{eq:tight} is obtained by
ruling out, along a subsequence,
 the 
two other cases
in the alternative established in
\cite[Lemma~I.1]{lions1984aihp}:  
\textit{vanishing}, that is,
\begin{align}\label{eq:vanish}
\sup_{z_0\in\R^2}\int_{|z-z_0|\leq R} f_k\, dz \to 0\,,\quad\forall R>0\,,
\end{align}
 and \textit{dichotomy}, 
 which implies, 
 modulo translations $\mathfrak T_{z_k}$,
 the existence of  $R_k>1$ such that
\begin{align}\label{eq:dicho}
&
\liminf \int_{|z|\leq  R_k} f_k \, dz 
 >0,
 \quad
\liminf \int_{|z|\geq 2R_k}f_k\, dz
 >0\,,
\\
&
\text{and }
\int_{R_k\leq |z|\leq 2 R_k}f_k\, dz \to 0\,.
\nonumber
\end{align} 
Contrary to \cite[Lemma~4.2]{melcher2014proc} for the easy-axis anisotropy,
in our case
vanishing \eqref{eq:vanish} does not seem to directly imply that $\widetilde H(\m^{(k)})\to 0$.
We rely instead on the upper bound \eqref{eq:up} and the characterization of near-minimizers in Proposition~\ref{p:lowE} to  discard vanishing.
Thanks to \eqref{eq:up} and the fact that $\m^{(k)}$ is a minimizing sequence, 
we know indeed that
$\m^{(k)}$ satisfies \eqref{eq:lowsatmsigma} for some absolute constant $K>0$ and large enough $k$.
According to Proposition~\ref{p:lowE}, there exists therefore a Möbius map $\Psi^{(k)}$ and $z_k\in\R^2$ such that
\begin{align*}
\int_{\R^2}|\nabla \m^{(k)}-\nabla\Psi^{(k)}|^2\, dz 
&
\leq \frac 14\,,
\\
\text{and}
\quad
\int_{|z-z_k|\leq 1}|\nabla\Psi^{(k)}|^2\, dz 
&
\geq 1\,.
\end{align*}
The second inequality follows from the fact that the concentration scale $\rho>0$ of the Möbius map $\Psi$ provided by Proposition~\ref{p:lowE} is arbitrarily small if $\sigma_0$ is small enough.
These two inequalities imply that
\begin{align*}
\int_{|z-z_k| \leq 1} f_k\, dz \geq \int_{|z-z_k|\leq 1}|\nabla \m^{(k)}|^2\, dz \geq \frac 14\,,
\end{align*}
thus ruling out vanishing \eqref{eq:vanish}.

The dichotomy case \eqref{eq:dicho} can be ruled out  as in \cite{melcher2014proc},
by using  the construction of \cite[Lemma~8]{doring2017cvpde} adapted to our setting.
It provides two maps
$\m^{(k,1)},\m^{(k,2)}\in \mathcal W$ such that
\begin{align}
&
\m^{(k)} =\m^{(k,1)} \text { in }B_{R_k}\,,\quad \m^{(k)}=\m^{(k,2)} \text{ outside }B_{2R_k}\,,
\nonumber
\\
&
\int_{|z|\geq R_k} |\nabla \m^{(k,1)}|^2 +(m^{(k,1)}_3)^2\, dx \longrightarrow 0\,,
\nonumber
\\
&
\int_{|z|\leq 2 R_k} |\nabla \m^{(k,2)}|^2 +(m^{(k,2)}_3)^2\, dx \longrightarrow 0\,.\label{eq:cutoff}
\end{align}
We briefly sketch that construction.
We select a radius $\rho_k\in [R_k,2R_k]$ such that
\begin{align*}
R_k\int_{\partial B_{\rho_k}}f_k\, d\mathcal H^1 \leq  \int_{R_k<|z|<2R_k}\!
f_k\, dz =\delta_k\to 0\,.
\end{align*}
This implies that the map 
$\theta\mapsto \hat\m^{(k)}(\theta)=\m^{(k)}(\rho_ke^{i\theta})$
satisfies
\begin{align*}
&
\frac{1}{2}\int_0^{2\pi}
|\partial_\theta \hat\m^{(k)}|^2\, d\theta
+ R_k^2
\int_0^{2\pi} (\hat m^{(k)}_3)^2\, d\theta \leq \delta_k\,.
\end{align*}
In particular, $\hat\m^{(k)}$ 
has small oscillation over $\mathbb S^1$, 
 its average must be close to $\mathbb S^1\times\lbrace 0\rbrace\subset\mathbb S^2$,
and we may find $\xi^{(k)}\in\mathbb S^1\times\lbrace 0\rbrace$ such that 
\begin{align*}
\sup_{\theta\in\mathbb S^1}\,
 |\hat\m^{(k)}(\theta)-\xi^{(k)}|^2 \leq C \delta_k\,.
\end{align*}
Then we define, using polar coordinates $x=re^{i\theta}$,
\begin{align*}
\m^{(k,1)}
&
=\begin{cases}
\m^{(k)} &\text{ in }B_{\rho_k},\\
\Pi\big(\xi^{(k)} +
\chi_1(r/\rho_k) \big(\hat\m^{(k)}(\theta)-\xi^{(k)}\big) \big)
&\text{ for }r\geq \rho_k\,,
\end{cases}
\\
\m^{(k,2)}
&
=\begin{cases}
\m^{(k)} &\text{ outside } B_{\rho_k},\\
\Pi\big(
\xi^{(k)}+
\chi_2(r/\rho_k)\big(\hat\m^{(k)}(\theta) - \xi^{(k)}\big)
\big)
&\text{ for }0<r< \rho_k\,,
\end{cases}
\end{align*}
where $\Pi(X)=X/|X|$ is the projection onto $\mathbb S^2$ 
and
 $\chi_1,\chi_2$ are smooth cut-off functions satisfying 
\begin{align*}
\mathbf 1_{r\leq 1}\leq \chi_1(r) \leq \mathbf 1_{r\leq 2},
\quad
\mathbf 1_{r\geq 1}\leq \chi_2(r)\leq\mathbf 1_{r\geq 1/2}\,,
\end{align*}
and it can be checked that these maps satisfy \eqref{eq:cutoff}.

The properties \eqref{eq:cutoff} of 
$\m^{(k,1)},\m^{(k,2)}$  imply that their integer-valued topological degrees satisfy 
\begin{align*}
Q(\m^{(k,1)})+Q(\m^{(k,2)})
=Q(\m^{(k)})=-1\,,
\end{align*}
and their energies satisfy
\begin{align}\label{eq:splitEsigma}
E_\sigma(\m^{(k)})=E_\sigma(\m^{(k,1)})+E_\sigma(\m^{(k,2)})+o(1)\,.
\end{align}
Since $\m^{(k)}$ is a minimizing sequence for $E_\sigma$ in $\mathcal W_{-1}$, 
and $E_\sigma\geq 0$, 
this together with the  inequality \eqref{eq:highdeg} 
implies, along a subsequence and for large enough $k$,
\begin{align*}
Q(\m^{(k,1)})=q_1 \in\lbrace  0,\pm 1\rbrace\,,
\qquad
Q(\m^{(k,2)})=q_2\in\lbrace 0,\pm 1\rbrace\,.
\end{align*}
Since $q_1+q_2=-1$, we have either $q_1=-1$ or $q_2=-1$.
Moreover, the first line in \eqref{eq:dicho}, the 
first inequality in \eqref{lower},
and the properties \eqref{eq:cutoff} of $\m^{(k,1)},\m^{(k,2)}$ imply, along a subsequence,
\begin{align*}
E(\m^{(k,1)})\to\mu_1>0,\quad E(\m^{(k,2)})\to\mu_2>0.
\end{align*}
If $q_1=-1$, we also have $\mu_1 \geq \inf_{\mathcal W_{-1}} E_\sigma =  \lim E_\sigma(\m^{(k)})$,
thus contradicting \eqref{eq:splitEsigma}.
And if $q_2=-1$ we have $\mu_2 \geq \lim E_\sigma(\m^{(k)})$
and again a contradiction. 
This shows that the dichotomy case \eqref{eq:dicho} cannot occur and concludes the proof of the tightness property \eqref{eq:tight}.
\end{proof}

\appendix

\section{Critical Sobolev spaces on the plane and on the sphere}\label{a:crit_sob_plane}

In this appendix we recall, 
for the readers' convenience, the identification
between functions with finite Dirichlet energy on the plane and on the sphere, via stereographic projection.
And, more generally, functions on $\R^n$ and $\mathbb S^n$ with finite $n$-energy for any $n\geq 2$.
That is, we compare the homogeneous Sobolev space
\begin{align}\label{eq:HRn}
\mathcal H(\R^n)
&
=\big\lbrace v\in W^{1,1}_{\loc}(\R^n)\colon \int_{\R^n}|\nabla v|^n\, dx <\infty\big\rbrace\,,
\end{align}
and the Sobolev space
 $W^{1,n}(\mathbb S^n)$.
As in \eqref{eq:stereo} we 
 define
\begin{align*}
\Phi\colon \R^n\cup\lbrace\infty\rbrace
&
\to\mathbb S^n\subset\R^{n+1}\,,
\\
x
&\mapsto \bigg( \frac{2x}{1+|x|^2},\frac{|x|^2-1}{1+|x|^2}\bigg)
\quad\text{if }x\in\R^n
\,,
\\
\infty
&
\mapsto \mathbf e_{n+1}=(0,\ldots,0,1)\,,
\end{align*}
whose inverse is the stereographic projection
\begin{align*}
\Phi^{-1}\colon \mathbb S^n
&
\to\R^n\,,
\\
\xi
&
\mapsto \frac{1}{1-\xi_{n+1}}(\xi_1,\ldots,\xi_n)
\quad\text{if }\xi\neq \mathbf e_{n+1}
\,,
\\
\mathbf e_{n+1} &\mapsto \infty\,.
\end{align*}
A crucial property of $\Phi$ is its conformality: 
\begin{align*}
\langle \partial_i\Phi,\partial_j\Phi\rangle =\frac{4}{(1+|x|^2)^2}\delta_{ij}\qquad\text{for }i,j\in\lbrace 1,\ldots,n\rbrace\,.
\end{align*}
(In general, a diffeomorphism between riemannian manifold is conformal if it induces a conformal change of metric, see e.g. \cite[Chapter~7]{reshetnyak94}.)
Given any $w\in C^1(\mathbb S^n)$, 
we have $v=w\circ\Phi\in C^1(\R^n)$, 
and the $n$-energy has the conformal invariance property 
\begin{align}\label{eq:conf_inv}
\int_{\mathbb S^n} |\nabla_T w|^n \, d\mathcal H^n = \int_{\R^n} |\nabla v|^n \, dx\,,
\end{align}
where $\nabla_T$ denotes the tangential gradient on $\mathbb S^n$.
Note that, given a measurable function $v\colon\R^n\to\R$, 
the function $w=v\circ\Phi^{-1}$
 is 
defined almost everywhere and measurable on $\mathbb S^n$. 
The following result is folklore, but we could not find a precise reference so we include a proof below.

\begin{prop}\label{p:crit_sob_Rn}
For any measurable $v\colon\R^n\to\R$, we have $v\in\mathcal H(\R^n)$ if and only if $w=v\circ\Phi^{-1}\in W^{1,n}(\mathbb S^n)$,
and the identity \eqref{eq:conf_inv} is satisfied.
\end{prop}

Before proving proving Proposition~\ref{p:crit_sob_Rn} we gather some consequences that are relevant to the setting of this article.
First, if $v\in\mathcal H(\R^n;\mathbb S^n)$, then $w=v\circ\Phi^{-1}\in W^{1,n}(\mathbb S^n,\mathbb S^n)$ 
 has a well-defined topological degree $Q\in\mathbb Z$, see e.g. \cite{brezis1995degree}.
Using the classical expression of that degree for $C^1$ maps and arguing by approximation gives the formula
 \begin{align*}
Q
&
=
\Xint{-}_{\mathbb S^n}\det(w,\partial_{\tau_1}w,\ldots,\partial_{\tau_n}w)\, d\mathcal H^{n}
\,,
\end{align*}
where $(\tau_1,\ldots,\tau_n)$ is any choice of direct orthonormal frame of the tangent space to $\mathbb S^n$.
 Coming back to $v$ and using again the conformality of $\Phi$ we obtain
 
\begin{cor}\label{c:Q}
For a sphere-valued map $v\in \mathcal H(\R^n;\mathbb S^n)$,
the quantity
\begin{align*}
Q(v)
&
=
\frac{1}{\mathcal H^{n}(\mathbb S^n)}\int_{\mathbb R^n} \det(v, \partial_1 v,\cdots, \partial_n v) \, dx  \, ,
\end{align*}
is a well-defined integer, which corresponds to the topological degree  of 
 $w=v\circ\Phi^{-1} \in W^{1,n}(\mathbb S^n;\mathbb S^n)$.
\end{cor}

Another consequence of Proposition~\ref{p:crit_sob_Rn} concerns
the vanishing mean oscillation property: functions in $W^{1,n}(\mathbb S^n)$ have vanishing mean oscillation at every point \cite[\S~I.2]{brezis1995degree}.
Applying this to the function $w$ at $\mathbf e_{n+1}$ translates, after a change of variable, into vanishing mean oscillation ``at $\infty$'' for
any function $v\in\mathcal H(\R^n)$, namely
\begin{align*}
\Xint{-}_{|x|\geq R} |v-v_R|\, \frac{dx}{(1+|x|^2)^n} \to 0
\qquad\text{as }R\to\infty\,,
\end{align*}
where the average is taken with respect to the measure $dx/(1+|x|^2)^n$, 
and  $v_R$ is the mean value
\begin{align*}
v_R =\Xint{-}_{|x|\geq R} v \,\frac{dx}{(1+|x|^2)^n}\,.
\end{align*}
In general, that mean value need not have a limit as $R\to\infty$.
In fact, we have $v_R\to v_\infty$ as $R\to\infty$  if and only if $\mathbf e_{n+1}$ is a Lebesgue point of $w$, with value $w(\mathbf e_{n+1})=v_\infty$.
Note that the function $R\mapsto v_R$ is continuous, so the set of possible limits of $v_{R_k}$ along subsequence $R_k\to\infty$ is connected.
Moreover, if we happen to have the additional information that
\begin{align*}
\int_{\mathbb R^n}f(v)\,\frac{ dx }{(1+|x|^2)^{\frac n2}} <\infty\,,
\end{align*}
for some $L$-Lipschitz function $f$, then we deduce
that $f(v_R)\to 0$ as $R\to\infty$, thanks to the inequalities
\begin{align*}
&
|f(v_R)|
=\Xint{-}_{|x|\geq R}|f(v_R)| \frac{dx}{(1+|x|^2)^n}
\\
&
\leq L \Xint{-}_{|x|\geq R} |v-v_R|\frac{dx}{(1+|x|^2)^n}
+\Xint{-}_{|x|\geq R} |f(v)|\frac{dx}{(1+|x|^2)^n}
\\
&
\leq L \Xint{-}_{|x|\geq R} |v-v_R|\frac{dx}{(1+|x|^2)^n}
+C\int_{|x|\geq R} |f(v)|\,\frac{ dx }{(1+|x|^2)^{\frac n2}} \,,
\end{align*}
where we used that
\begin{align*}
\frac{1}{(1+R^2)^{\frac n2}}\leq C \int_{|x|\geq R}\frac{dx}{(1+|x|^2)^n}\,,
\end{align*}
for some constant $C>0$ depending on $n$.
If $v$ takes values into a closed set $X\subset\R^k$, applying this to $f=\dist(\cdot,X)$ shows that $\dist(v_R,X)\to 0$ as $R\to\infty$.
And if we know in addition that $g(v)$ is integrable on $\R^n$ for some Lipschitz function $g$, then all possible limits of subsequences $v_{R_k}$ must belong to a single connected component of $X\cap g^{-1}(\lbrace 0\rbrace)$.
If all such connected components are points, then $v_R$ converges and $\mathbf e_{n+1}$ is a Lebesgue point of $w$. 
Applying this to $n=2$, $X=\mathbb S^2\subset\R^3$ and $g(x)=1-x_3$ or $1-x_3^2$ for $x\in\mathbb S^2$, we deduce the following.

\begin{cor}\label{c:lebesgue_infty}
    For a map $\m  \in \mathcal H(\R^2;\mathbb S^2)$, if either $1-m_3^2$ or $1-m_3$ is integrable on $\R^2$, then $\mathbf e_{3}$ is a Lebesgue point of $w=\m\circ\Phi^{-1} \in W^{1,2}(\mathbb S^2;\mathbb S^2)$.
\end{cor}

We finally turn to the proof of Proposition~\ref{p:crit_sob_Rn}.

\begin{proof}[Proof of Proposition~\ref{p:crit_sob_Rn}]
First we show that, if $w\in W^{1,n}(\mathbb S^{n})$, then $v=w\circ\Phi$ belongs to $\mathcal H(\R^n)$ and \eqref{eq:conf_inv} is satisfied.
This follows  
by
 approximating
$w$ with smooth functions $w_\e\in C^1(\mathbb S^n)$
to which we apply \eqref{eq:conf_inv} and pass to the limit.
To make sure that $v_\e=w_\e\circ\Phi$ converges in $L^1_{\loc}(\R^n)$ to $v=w\circ\Phi$
 we use for instance
 the identity
 \begin{align}\label{eq:wL1_changevar}
\int_{\mathbb S^n}u(y) \,d\mathcal H^n (y)
=
\int_{\R^n} u \circ\Phi(x) \frac{2^n dx}{(1+|x|^2)^n}\,,
\end{align}
valid for any measurable $u\colon \mathbb S^n\to [0,\infty]$,
applied to $u=|w-w_\e|$.

Reciprocally, let us 
assume that $v\in\mathcal H(\R^n)$, and prove that $w=v\circ\Phi^{-1}$ belongs to $W^{1,n}(\mathbb S^n)$.
From localized versions of the identities \eqref{eq:conf_inv}-\eqref{eq:wL1_changevar} we see that
\begin{align*}
w\in W^{1,n}_{\loc}(\mathbb S^n\setminus \lbrace\mathbf e_{n+1}\rbrace)
\quad\text{and}
\quad
\nabla_T w\in L^n( \mathbb S^n\setminus \lbrace\mathbf e_{n+1}\rbrace),
\end{align*}
where $\nabla_T w$ is the distributional gradient of $w$ in
 $\mathbb S^n\setminus \lbrace\mathbf e_{n+1}\rbrace$.
In order to conclude that   $w\in W^{1,n}(\mathbb S^n)$, it suffices to show that $w\in L^n(\mathbb S^n)$.
This implies indeed that $w\in W^{1,n}(\mathbb S^n\setminus \lbrace \mathbf e_{n+1}\rbrace)$ and therefore $w\in W^{1,n}(\mathbb S^n)$
since, using the terminology of \cite{koskela99},
 points are removable for Sobolev spaces in dimension $n\geq 2$.
We recall here the proof of that fact: working in a local chart, this amounts to showing that, if $\Omega=(-1,1)^n$ and 
$u\in W^{1,p}(\Omega\setminus\lbrace 0\rbrace)$ for some $p\geq 1$, 
then the distributional gradient of $u$ on $\Omega$ is equal to its distributional gradient on $\Omega\setminus \lbrace 0\rbrace$.
To check this, note that for almost every $x' = (x_2,\ldots,x_n)\in \Omega'= (-1,1)^{n-1}\setminus\lbrace 0\rbrace$ we have $u(\cdot,x')\in W^{1,p}((-1,1))$ with distributional derivative given by $\partial_1 u(\cdot,x')\in L^p((-1,1))$,
so  for any   $\varphi\in C_c^\infty(\Omega)$ we can compute
\begin{align*}
\langle \partial_1 u,\varphi\rangle 
&
=- \int_{\Omega'} \int_{-1}^1  u\,\partial_1\varphi\, dx_1 dx' 
=\int_{\Omega'}\int_{-1}^1 \varphi\, \partial_1 u   \, dx_1 dx' 
\\
&
=\int_{\Omega\setminus\lbrace 0\rbrace}\varphi\, \partial_1 u  \, dx\,,
\end{align*}
and similarly for the other derivatives.
In view of \eqref{eq:wL1_changevar}, the fact that $w\in L^n(\mathbb S^n)$ 
is ensured by Lemma~\ref{l:vLp} below.
\end{proof}

\begin{rem}
It would have been tempting to prove the second implication ($v\in \mathcal H\Rightarrow w\in W^{1,n}$) by approximation
with smooth functions, as 
for the first implication.
But there is an additional difficulty: 
while 
 $w\in C^1(\mathbb S^n)$ implies $v=w\circ\Phi\in C^1\cap\mathcal H(\R^n)$,
the converse implication it not true.
For instance, the function $v(x)=\ln(1+\ln(1+|x|^2))$
belongs to $C^1\cap\mathcal H(\R^n)$, 
but $w=v\circ\Phi^{-1}$ is not continuous at $\mathbf e_{n+1}\in\mathbb S^n$. (In fact $\mathbf e_{n+1}$ is not even a Lebesgue point of $w$.)
That is why we argued
 differently.
\end{rem}

\begin{lem}\label{l:vLp}
For all $v\in\mathcal H(\R^n)$  we have
\begin{align*}
\int_{\R^n}|v|^n \frac{dx}{(1+|x|^2)^n} 
\leq c  \int_{\R^n}|\nabla v|^n\, dx   + c  \int_{B_1}|v|^n\, dx  \,,
\end{align*}
for some constant $c>0$ depending on $n$.
\end{lem}
\begin{proof}[Proof of Lemma~\ref{l:vLp}]
We introduce the notation
\begin{align*}
A(v)=  \int_{\R^n}|\nabla v|^n\, dx  + c  \int_{B_1}|v|^n\, dx  \,,
\end{align*}
and denote by $c>0$ a generic constant depending on $n$, which may change from line to line.
By Sobolev inequality in  $B_1\subset\R^n$ we have
\begin{align*}
\int_{B_1}|v|^n\, dx \leq c A(v)\,,
\end{align*}
so it suffices to show that
\begin{align*}
\int_{|x|\geq 1} |v|^n\frac{dx}{|x|^{2n}}
\leq c A(v)
\,.
\end{align*}
To that end we consider the function 
\begin{align*}
f(r)=\Xint{-}_{\mathbb S^{n-1}}v(r\omega)\,d\mathcal H^{n-1}(\omega)\quad\forall r\geq 1\,.
\end{align*}
We fix $r_0\in (1/2,1)$ such that
\begin{align*}
|f(r_0)| \leq c\int_{B_1}|v|\, dx \,.
\end{align*}
For all $R\geq 1$ we have
\begin{align*}
|f(R)-f(r_0)|
&
 \leq \int_{r_0}^R \Xint{-}_{\mathbb S^{n-1}}|\partial_r v|(r\omega) \, d\mathcal H^{n-1}(\omega) \,dr
\\
&
\leq \int_{B_R\setminus B_{r_0}}|\nabla v|\, \frac{dx}{|x|^{n-1}}
\\
&
\leq \bigg(\int_{\R^n}|\nabla v|^n\, dx \bigg)^{\frac 1n}
\bigg(\int_{B_R\setminus B_{r_0}} \frac{dx}{|x|^n}\bigg)^{1-\frac 1n}
\\
&
\leq c A(v)^{\frac 1n} \ln(2R/r_0)\,,
\end{align*}
and therefore
\begin{align*}
|f(R)|^n \leq c A(v)\ln^n(4R)\,,\quad\forall R\geq 1\,.
\end{align*}
By Sobolev embedding $W^{1,n}(\mathbb S^{n-1})\subset L^\infty(\mathbb S^{n-1})$ and definition of $f(r)$, 
for all $r\geq 1$ and $\omega\in\mathbb S^{n-1}$ we have
\begin{align*}
|v(r\omega)|^n \leq c |f(r)|^n  + 
c r^n   \int_{\mathbb S^{n-1}}|\nabla v|^n(r\tilde\omega)\,d\mathcal H^{n-1}(\tilde\omega) \,.
\end{align*}
Using the two last inequalities we deduce
\begin{align*}
\int_{|x|\geq 1}|v|^n\, \frac{dx}{|x|^{2n}}
&
=
\int_{\mathbb S^{n-1}} \int_{1}^\infty |v(r\omega)|^n\,\frac{dr}{r^{n+1}}\, d\mathcal H^{n-1}(\omega) 
\\
&
\leq c\int_1^{\infty}|f(r)|^{n}\, \frac{dr}{r^{n+1}} 
+ c 
\int_{|x|\geq 1}|\nabla v|^n\, \frac{dx}{|x|^n} 
\\
&
\leq c A(v) \,,
\end{align*}
and this concludes the proof.
\end{proof}

\bibliographystyle{acm}
\bibliography{ref_bimeron.bib}

\begin{thebibliography}{10}

\bibitem{abramowitz1972handbook}
{\sc Abramowitz, M., and Stegun, I.~A.}
\newblock Handbook of mathematical functions with formulas, graphs, and
  mathematical tables. 10th printing, with corrections.
\newblock National {Bureau} of {Standards}. {A} {Wiley}-{Interscience}
  {Publication}. {New} {York} etc.: {John} {Wiley} \& {Sons} (1972), 1972.

\bibitem{bachmann2023prb}
{\sc Bachmann, D., Lianeris, M., and Komineas, S.}
\newblock Meron configurations in easy-plane chiral magnets.
\newblock {\em Phys. Rev. B 108\/} (Jul 2023), 014402.

\bibitem{bacho-melcher}
{\sc Bacho, G., and Melcher, C.}
\newblock Bimeron configuration in thin chiral ferromagnetic films with large
  easy-plane anisotropy.
\newblock In preparation.

\bibitem{bernand2021arma}
{\sc Bernand-Mantel, A., Muratov, C.~B., and Simon, T.~M.}
\newblock A quantitative description of skyrmions in ultrathin ferromagnetic
  films and rigidity of degree {{\(\pm 1\)}} harmonic maps from
  {{\(\mathbb{R}^2\)}} to {{\(\mathbb{S}^2\)}}.
\newblock {\em Arch. Ration. Mech. Anal. 239}, 1 (2021), 219--299.

\bibitem{brezis1995degree}
{\sc Br{\'e}zis, H., and Nirenberg, L.}
\newblock Degree theory of {BMO}. {I}: {Compact} manifolds without boundaries.
\newblock {\em Sel. Math., New Ser. 1}, 2 (1995), 197--263.

\bibitem{doring2017cvpde}
{\sc D{\"o}ring, L., and Melcher, C.}
\newblock Compactness results for static and dynamic chiral skyrmions near the
  conformal limit.
\newblock {\em Calc. Var. Partial Differ. Equ. 56}, 3 (2017), 30.
\newblock Id/No 60.

\bibitem{esteban1986commun}
{\sc Esteban, M.~J.}
\newblock A direct variational approach to {Skyrme}'s model for meson fields.
\newblock {\em Commun. Math. Phys. 105\/} (1986), 571--591.

\bibitem{gobel2019magnetic}
{\sc G{\"o}bel, B., Mook, A., Henk, J., Mertig, I., and Tretiakov, O.~A.}
\newblock Magnetic bimerons as skyrmion analogues in in-plane magnets.
\newblock {\em Phys. Rev. B 99}, 6 (2019), 060407.

\bibitem{Hang-Lin}
{\sc Hang, F.~B., and Lin, F.~H.}
\newblock Static theory for planar ferromagnets and antiferromagnets.
\newblock {\em Acta Math. Sin. (Engl. Ser.) 17}, 4 (2001), 541--580.

\bibitem{hirsch2022}
{\sc Hirsch, J., and Zemas, K.}
\newblock A note on a rigidity estimate for degree {{\(\pm 1\)}} conformal maps
  on {{\(\mathbb{S}^2\)}}.
\newblock {\em Bull. Lond. Math. Soc. 54}, 1 (2022), 256--263.

\bibitem{Ignat}
{\sc Ignat, R.}
\newblock Singularities of divergence-free vector fields with values into
  {$S^1$} or {$S^2$}. {A}pplications to micromagnetics.
\newblock {\em Confluentes Math. 4}, 3 (2012), 1230001, 80.

\bibitem{IgOf}
{\sc Ignat, R., and L'Official, F.}
\newblock Renormalised energy between boundary vortices in thin-film
  micromagnetics with {D}zyaloshinskii-{M}oriya interaction.
\newblock {\em Nonlinear Anal. 250\/} (2025), Paper No. 113622, 26.

\bibitem{koskela99}
{\sc Koskela, P.}
\newblock Removable sets for {Sobolev} spaces.
\newblock {\em Ark. Mat. 37}, 2 (1999), 291--304.

\bibitem{lemaire1978jdg}
{\sc Lemaire, L.}
\newblock Applications harmoniques de surfaces riemanniennes.
\newblock {\em J. Differ. Geom. 13\/} (1978), 51--78.

\bibitem{li2018jfa}
{\sc Li, X., and Melcher, C.}
\newblock Stability of axisymmetric chiral skyrmions.
\newblock {\em J. Funct. Anal. 275}, 10 (2018), 2817--2844.

\bibitem{lin2004cpam}
{\sc Lin, F., and Yang, Y.}
\newblock Existence of two-dimensional {Skyrmions} via the
  concentration-compactness method.
\newblock {\em Commun. Pure Appl. Math. 57}, 10 (2004), 1332--1351.

\bibitem{lions1984aihp}
{\sc Lions, P.-L.}
\newblock The concentration-compactness principle in the calculus of
  variations. {The} locally compact case. {I}.
\newblock {\em Ann. Inst. Henri Poincar{\'e}, Anal. Non Lin{\'e}aire 1\/}
  (1984), 109--145.

\bibitem{melcher2014proc}
{\sc Melcher, C.}
\newblock Chiral skyrmions in the plane.
\newblock {\em Proc. R. Soc. Lond., Ser. A, Math. Phys. Eng. Sci. 470}, 2172
  (2014), 17.
\newblock Id/No 20140394.

\bibitem{moser1971iumj}
{\sc Moser, J.}
\newblock A sharp form of an inequality by {Trudinger}.
\newblock {\em Indiana Univ. Math. J. 20\/} (1971), 1077--1092.

\bibitem{reshetnyak94}
{\sc Reshetnyak, Y.~G.}
\newblock {\em Stability theorems in geometry and analysis}, rev. and updated
  transl.~ed., vol.~304 of {\em Math. Appl., Dordr.}
\newblock Dordrecht: Kluwer Academic Publishers, 1994.

\bibitem{rupflin23}
{\sc Rupflin, M.}
\newblock Sharp quantitative rigidity results for maps from {$ \mathbb S^2$} to
  {$ \mathbb S^2$} of general degree.
\newblock {\em arXiv:2305.17045\/} (2023).

\bibitem{struwe2008varmeth}
{\sc Struwe, M.}
\newblock {\em Variational methods. {Applications} to nonlinear partial
  differential equations and {Hamiltonian} systems.}, 4th ed.~ed., vol.~34 of
  {\em Ergeb. Math. Grenzgeb., 3. Folge}.
\newblock Berlin: Springer, 2008.

\bibitem{topping2023}
{\sc Topping, P.~M.}
\newblock A rigidity estimate for maps from {{\(S^2\)}} to {{\(S^2\)}} via the
  harmonic map flow.
\newblock {\em Bull. Lond. Math. Soc. 55}, 1 (2023), 338--343.

\end{thebibliography}

\end{document}